\documentclass[11pt,oneside,reqno]{amsart}

  \usepackage[utf8]{inputenc}
  \usepackage[T1]{fontenc}
  \usepackage[russian,ngerman,english]{babel}
  \usepackage{lmodern}

  \usepackage{hyperref}
  \usepackage{amsmath}
  \usepackage{amsthm}
  \usepackage{amsfonts}
  \usepackage{amssymb}
  \usepackage{amscd}
  \usepackage{amsbsy}

  \usepackage{pdfpages}
  \usepackage{fancyhdr}
  \usepackage{geometry}
  \usepackage{setspace}
  \usepackage{xcolor}
  \usepackage{pifont}

  \usepackage{graphicx}
  \usepackage{tabularx}
  \usepackage{epsfig}
  \usepackage[all]{xy}
  \usepackage{tikz}
  \usetikzlibrary{matrix}

  \usepackage{fixmath}

  \usepackage{appendix}

  \usepackage{enumerate}

  \usepackage[sort, numbers]{natbib}
  \usepackage{bibentry}

  \newtheorem{theorem}{Theorem}[section]

  \newtheorem{lemma}[theorem]{Lemma}

  \theoremstyle{definition}
  \newtheorem{definition}[theorem]{Definition}

  \newtheorem*{remark}{Remark}

  \newtheorem{example}[theorem]{Example}

  \numberwithin{equation}{section}

  \newcommand{\R}{{\mathbb R}}

  \newcommand{\E}{{\mathsf E}}

  \newcommand{\supp}{{\operatorname{supp}}}

  \newcommand{\tr}{\operatorname{tr}}

\DeclareMathOperator{\Int}{int}
\DeclareMathOperator{\Mat}{Mat}

  \newcommand{\cD}{{\mathcal{D}}}

  \newcommand{\cR}{{\mathcal{R}}}
  
  \newcommand{\cT}{{\mathcal{T}}}

  \newcommand{\fS}{{\mathfrak{S}}}

  \newcommand{\e}{{\varepsilon}}

  \newcommand{\z}{\zeta}
  \newcommand{\g}{\gamma}
  
  \renewcommand{\l}{\lambda}

  \renewcommand{\d}{\delta}
  \renewcommand{\t}{{\theta}}

  \newcommand{\dd}{\mathrm{d}}
  
  \renewcommand{\Re}{\operatorname{Re}}
  \renewcommand{\Im}{\operatorname{Im}}
  
  \newcommand{\dist}{\operatorname{dist}}

  \newcommand{\SL}{\text{SL}}

  \newcommand{\ess}{\text{\rm{ess}}}

  \newcommand{\bbN}{\mathbb{N}}
  \newcommand{\bbZ}{\mathbb{Z}}
  
  \newcommand{\bbR}{\mathbb{R}}
  \newcommand{\bbC}{\mathbb{C}}

  \newcommand{\AC}{\mathrm{AC}}
  \newcommand{\loc}{\mathrm{loc}}

  \allowdisplaybreaks

  \title[Asymptotics for Christoffel functions]{Asymptotics for Christoffel functions associated to continuum Schr\"odinger operators}

\author[B.\ Eichinger]{Benjamin Eichinger}
\address{Institute for Analysis and Scientific Computing, Vienna University of Technolog, Wien, A-1040, Austria}

\email{benjamin.eichinger@tuwien.ac.at}

\thanks{This work was supported by the Austrian Science Fund FWF, project no: P33885}

\begin{document}
	\maketitle
	
	\begin{abstract}
	We prove asymptotics of the Christoffel function, $\l_L(\xi)$, of a continuum Schr\"odinger operator for points in the interior of the essential spectrum under some mild conditions on the spectral measure. It is shown that $L\l_L(\xi)$ has a limit and that this limit is given by the Radon--Nikodym derivative of the spectral measure with respect to the Martin measure. Combining this with a recently developed local criterion for universality limits at scale $\l_L(\xi)$, we compute universality limits for continuum Schr\"odinger operators at scale $L$ and obtain clock spacing of the eigenvalues of the finite range truncations.
	\end{abstract}

\section{Introduction}

The goal of this paper is to derive asymptotics for Christoffel functions of continuum Schr\"odinger operators. It is natural for this topic to work in the half-line setting, so our Schr\"odinger operators are unbounded self-adjoint operators, $H_V$, on $L^2((0,\infty))$, corresponding formally to the differential expression
\begin{align*}
- \frac{d^2}{dx^2}+V.
\end{align*}
We always assume that the potential $V$ is real-valued and uniformly locally integrable, i.e.
\begin{equation}\label{L1locunif}
\sup_{x \ge 0} \int_x^{x+1} \lvert V(t) \rvert\dd t < \infty.
\end{equation}
In particular, $0$ is a regular endpoint and $+\infty$ is a limit point endpoint in the sense of Weyl. We set a Neumann boundary condition at $0$,  so the domain of $H_V$ is 
\[
D(H_{V}) = \big\{ f \in L^2((0,\infty)) \mid f,f' \in \AC_\loc([0,\infty)), -f'' + V f \in L^2((0,\infty)),\, f'(0) = 0\big\}
\]
where $\AC_\loc([0,\infty))$ denotes the set of functions which are absolutely continuous on all bounded intervals.

For any $z\in\bbC$ the  Neumann solution, $v(x,z)$, is the solution of the initial value problem 
\begin{align}\label{defnv}
-\partial_x^2v(x,z)+V(x)v(x,z)=z v(x,z),\quad v(0,z)=1,\ \partial_xv(0,z)=0.
\end{align}
The Christoffel function is defined by 
\begin{align}\label{def:Christoffel}
\l_L(z)=\bigg(\int_0^L|v(x,z)|^2dx\bigg)^{-1},\quad z\in\bbC,L\geq 0.
\end{align}
As a function of $L$, it measures the growth rate of eigensolutions, which is known to be an important quantity in spectral theory. For instance growth rates of eigensolutions are used in subordinacy theory developed by Gilbert and Pearson \cite{GilbertPearson}, or by Last and Simon \cite{LastSimonInvent99}, to describe the absolutely continuous spectrum of $H_V$. In our main result, Theorem \ref{thm1} below, we will prove asymptotics for $\l_L(\xi)$  as $L\to\infty$ and as a consequence obtain universality limits for Christoffel--Darboux kernels of continuum Schr\"odinger operators and clock spacing of the eigenvalues of finite range truncations of $H_V$.  Asymptotics of $\l_L$ as well as universality limits and zero spacing of eigenvalues has received much attention in the recent years, see \cite{AvilaLastSimon,BessonovMNT,BreuerWeissman,ELS,GubkinMNT,LevinLubinsky08,LubJFA,LubinskyAnnals,MNT,MaltsevCMP,Mitkovski} for a partial list of references.

 In order to formulate Theorem \ref{thm1} we need to recall the construction of a maximal spectral measure using Weyl theory and the concept of the Martin function from potential theory. 

Since $\infty$ is a limit point endpoint, there is (up to a scalar multiple unique) $\psi(x,z)$ satisfying
\begin{align}\label{eq:WeylSol}
-\partial_x^2\psi(x,z)+V(x)\psi(x,z)=z\psi(x,z)
\end{align}
and $\psi\in L^2((0,\infty)),$ which is called the Weyl solution at $\infty$. On the upper half-plane $\bbC_+$, the Weyl $m$-function is defined by 
\begin{align}\label{eq:mfunction}
m(z)=-\frac{\psi(0,z)}{\partial_x\psi(0,z)}.
\end{align} 
The function $m$ is a Herglotz function, i.e., it maps $\bbC_+$ analytically into itself. It is a general fact, that Herglotz functions admit an integral representation. That is, there exist $a\geq 0,b\in\bbR$ and a positive Borel measure $\mu$, with
\[
\int\frac{d\mu(\xi)}{1+\xi^2}<\infty
\]
such that 
\[
m(z)=az+b+\int_\bbR\left(\frac{1}{\xi-z}-\frac{\xi}{1+\xi^2}\right)d\mu(\xi).
\]
From the perspective of operator theory, the measure $\mu$ represents a maximal spectral measure of $H_V$. Let $d\mu=f_\mu dx+d\mu_s$ be the Lebesgue decomposition of $\mu$ with respect to the Lebesgue measure. 

In \cite{EichLuk} a theory of Stahl--Totik regularity was developed for continuum Schr\"odinger operators. To introduce this theory we will use some standard objects from potential theory which can be found in \cite{RansfordPotTheorie, Classpotential}. Let $\E=\sigma_{\ess}(H_V)$ and $\Omega=\bbC\setminus \E$. For potentials $V$ satisfying \eqref{L1locunif}, $\E$ is bounded from below but not from above. Therefore, one can show that the cone of positive harmonic functions in $\Omega$ which vanish quasi-everywhere (q.e.) on $\E$ is one-dimensional. Elements of this cone are called Martin functions of $\Omega$ at $\infty$. For an excellent survey on the Martin theory for Denjoy domains we refer to \cite{GardSj09}. In \cite{EichLuk} it is shown that for any such Martin function 
\[
\lim\limits_{z\to-\infty}\frac{M(z)}{\sqrt{-z}}>0.
\]
Existence of the limit follows by standard arguments for positive harmonic functions. The important point of the above statement is that the limit is positive. It allows to normalize at $\infty$ and we obtain a unique Martin function, $M_\E(z)$, such that the limit above is equal to $1$.  Due to \cite[Theorem 1.1]{EichLuk} there exists $a_\E \in \bbR$ such that the Martin function has the asymptotic behavior
\begin{align*}
M_\E(z)=\Re\left(\sqrt{-z}+\frac{a_\E}{2\sqrt{-z}}\right)+o\left(\frac{1}{\sqrt{|z|}}\right),
\end{align*}
as $z \to \infty$, $\arg z \in [\delta,2\pi - \delta]$, for any $\delta > 0$. This higher asymptotic expansion is then used to characterize regularity in the sense of Stahl and Totik. A potential $V$ satisfying \eqref{L1locunif} is called Stahl--Totik regular if 
\begin{align}\label{eq:StahlTotik1}
\lim\limits_{L\to\infty}\frac{1}{L}\int_0^LV(s)ds=a_\E.
\end{align}
The Martin function can be extended to a subharmonic function on $\bbC$ and thus its distributional Laplacian defines a positive measure. We call 
$$
\rho_\E=\frac{1}{2\pi}\Delta M_\E
$$ 
the Martin measure of the domain $\bbC\setminus\E$. It plays the role of the equilibrium measure from the theory of orthogonal polynomials for compactly supported measures.  Again we write its Lebesgue decomposition $d\rho_\E(\xi)=f_\E(\xi) d\xi+d\rho_{\E,s}(\xi)$.

Assuming regularity, we are able to characterize the asymptotic behavior of $\l_L$ at interior points of $\sigma_{\ess}(H_V)$.
\begin{theorem}\label{thm1}
	Let $V$ be a Stahl--Totik regular potential such that $\E=\sigma_{\ess}(H_V)$ is Dirichlet regular and $\mu$ the corresponding spectral measure. Let $I\subset\Int(\E)$ be a closed interval such that $\mu$ is absolutely continuous in a neighborhood of $I$ and its density $f_\mu$ is positive and continuous at every point of $I$. Then we have 
	\begin{align}\label{eq:CFasymp}
	\lim\limits_{L\to\infty}L\l_{L}(\xi)=\frac{f_\mu(\xi)}{f_\E(\xi)},
	\end{align}
uniformly for $\xi\in I$.
\end{theorem}

Stahl--Totik regularity is a quite general property. For sufficiently nice sets $\E$, regularity follows from
\begin{align}\label{eq:regCrit}
f_\mu>0\quad \text{ Lebesgue a.e. on }\E.
\end{align}
To be precise, by the Widom criterion \cite[Theorem 1.8]{EichLuk} $V$ is regular if the harmonic measure of $\bbC\setminus\E$ is absolutely continuous with respect to $\mu$. Thus, if $\E$ is such that the harmonic measure is mutually absolutely continuous with the Lebesgue measure restricted to $\E$, \eqref{eq:regCrit} implies Stahl--Totik regularity. By \cite{SodYud97}, mutual absolute continuity of the harmonic and the Lebesgue measure holds for regular Parreau--Widom sets. These sets are well studied in inverse spectral theory \cite{SodinYudSturmLiouville,BDGL,EVY}. Every set which is homogeneous in the sense of Carleson, i.e., sets $\E$ for which there exists $\tau>0$ so that 
\[
|\E\cap[\xi_0-\e,\xi_0+\e]|\geq \tau\e,\quad \forall \xi_0\in \E, \forall \e\in (0,1].
\] 
is a regular Parreau-Widom set \cite{JonesMarshall}. In particular, for finite gap sets or for spectra of Schr\"odinger operators with periodic potentials \eqref{eq:regCrit} implies that $V$ is Stahl-Totik regular. 

Christoffel functions are well studied in the setting of orthogonal polynomials. In this case, the Christoffel function is defined similar to \eqref{def:Christoffel}, but the Neumann solution at $L$ is substituted by the orthonormal polynomial of degree $n$.  For compactly supported measures, typical results show that under certain assumptions
\begin{align}\label{CDPoly}
\lim\limits_{n\to\infty}n\l_{n}(\xi)=\frac{f_\mu(\xi)}{f_\E(\xi)},
\end{align}
where $\l_{n}(\xi)$ is the Christoffel function associated to the orthonormal polynomials and $f_\E(\xi)$ denotes the density of the equilibrium measure. 
A fundamental result of M\'at\'e--Nevai--Totik \cite{MNT} shows \eqref{CDPoly} for the case $\E=[-2,2]$. More precisely, it is shown that \eqref{CDPoly} holds provided that $\mu$ is Stahl--Totik regular on $[-2,2]$, $f_\mu(\xi)>0$, $\log f_\mu$ is integrable in a neighborhood of $\xi$, and $\xi$ is a Lebesgue point of both the measure $\mu$ and the Szeg\H o function associated to $f_\mu$. This has been extended by Totik to arbitrary compact sets by using the polynomial preimage method \cite{Totik2000AssympChris}. 

Our approach is inspired by a method used by Simon and we obtain a full analog for continuum Schr\"odinger operators of all results in \cite{SimonTwoExt08}. Let us mention that the assumptions in \cite{SimonTwoExt08} or in Theorem \ref{thm1} are stronger than the ones in \cite{MNT}. However, the conclusion is also stronger, since uniformity in \eqref{eq:CFasymp} require continuity of $f_\mu$, see also \cite{TotikUniversality09}. It as an interesting question if our method could also be used to prove \eqref{eq:CFasymp} under Lebesgue point and local Szeg\H o conditions as used by M\'at\'e--Nevai--Totik.

Limits of Christoffel functions for continuum Schr\"odinger operators were first studied by Maltsev in \cite{MaltsevCMP}.  At that time the notion of regularity for continuum Schr\"odinger operators was not available and Maltsev proved \eqref{eq:CFasymp} for potentials $V=\mathring V +\tilde V$, where $\mathring V$ is a periodic continuous potential, $\tilde V$ is so that $\sigma_{\ess}(V)=\sigma_{\ess}(\mathring V)$ and  $\tilde V$ is C\'esaro decaying, i.e., 
\begin{align}\label{eq:CS}
\lim\limits_{L\to \infty}\frac{1}{L}\int_0^L|\tilde V(x)|dx=0.
\end{align}
Thus, our result generalize \cite{MaltsevCMP} in several directions. First of all, if $\E$ is the spectrum of a continuum Schr\"odinger operator, then generically there is no periodic potential so that that the essential spectrum of the associated operator is $\E$. Moreover, even if $\E$ is the spectrum of a periodic Schr\"odinger operator, a regular potential does not necessarily satisfy \eqref{eq:CS}. A counterexample can be found even in the simplest case $\E=[0,\infty)$ with $\mathring V \equiv 0$. It is shown in \cite[Example 1.13]{EichLuk} that the potential defined piecewise by $V(x) = (-1)^{\lfloor 2n(x-n) \rfloor}$ on $x \in [n-1,n)$ for integer $n$ is regular with $\sigma_\ess(L_V) = [0,\infty)$, but \eqref{eq:CS} does not hold. On the other hand, since periodic potentials are regular, it follows from \eqref{eq:StahlTotik1} and \eqref{eq:CS} that the potentials considered in \cite{MaltsevCMP} are Stahl--Totik regular. 

We turn to applications of our main theorem. The Christoffel-Darboux kernel is defined by 
\[
K_L(z,w)=\int_0^Lv(x,z)\overline{v(x,w)}dx.
\]
For $\xi\in\bbR$, bulk universality limits are double scaling limits of the type
\begin{align*}
\lim\limits_{L\to\infty}\frac{K_L\left(\xi+\frac{z}{\tau_L(\xi)},\xi+\frac{w}{\tau_L(\xi)}\right)}{K_L(\xi,\xi)}=\frac{\sin(\eta\pi(z-\overline{w}))}{\eta\pi (z-\overline{w})}.
\end{align*}
Universality limits were often studied at explicit polynomial scales. Recent results suggest that universality limits at such scales is a combination of two different phenomena. One is universality at scale $\l_{L}(\xi)$, studied in great generality in \cite{ELS}; the other is the explicit asymptotics of $\l_{L}(\xi)$ now provided by Theorem \ref{thm1}. Thus, we can combine Theorem \ref{thm1} with the results from \cite{ELS} to obtain universality limits at scale $L$.

\begin{theorem}\label{thm2}
	With the assumptions of Theorem \ref{thm1} we have 
	\begin{align}\label{eq:universality}
	\lim\limits_{L\to\infty}\frac{K_L\left(\xi+\frac{z}{L},\xi+\frac{w}{L}\right)}{K_L(\xi,\xi)}=\frac{\sin(\pi f_\E(\xi)(z-\overline w))}{\pi f_\E(\xi)(z-\overline w)},
	\end{align}
	uniformly for $\xi\in I$.
\end{theorem}

As a consequence of Theorem \ref{thm2} we will obtain asymptotic equal eigenvalue spacing of the eigenvalues of the finite range truncations. It is a common scheme already observed by Wigner for random matrix ensembles that the global asymptotic distribution of the eigenvalues depends on the particular model, however the local microscopic scale exhibits universal behavior. 

For any $L>0$, let $\nu_L$ denote the zero counting measure for $\partial_Lv(L,\cdot)$ divided by $L$
\[
\nu_L=\frac{1}{L}\sum_{\xi:\ \partial_Lv(L,\xi)=0}\delta_\xi.
\]
The measure $\nu_L$ is intimately related to the eigenvalues of the finite range truncation of $H_V$. Namely, if $H^{L}_V$ denotes the restriction of $H_V$ onto $(0,L)$ with Neumann boundary condition at $L$, then $H^L_{V}$ has purely discrete spectrum given by the zeros of $\partial_Lv(L,\cdot)$. 

The global distribution is given by the Martin measure $\rho_\E$. That is, regularity of $V$ implies that $\nu_L$ has a weak-$*$ limit and that this limit is given by $\rho_\E$.  However, it follows from the Freud--Levin theorem \cite{Freud,LevinLubinsky08} that \eqref{eq:universality} implies equal eigenvalue spacing at scale $1/L$.

For $L>0$ and $\xi\in I$, we denote by $\xi_j^{L}(\xi)$ for $j\in \bbZ$ the zeros of $\partial_Lv(L,\cdot)$ counted from $\xi$, i.e.,
\[
\dots < \xi_{-2}^{L}(\xi) <  \xi_{-1}^{L}(\xi) < \xi \le  \xi_{0}^{L}(\xi) <  \xi_{1}^{L}(\xi) <  \dots
\]
with no zeros of $\partial_Lv(L,\cdot)$ between $\xi_j^{L}$ and $\xi_{j+1}^{L}$.
\begin{theorem}\label{thm3}
	With the assumptions of Theorem \ref{thm1} we have that the zeros of $\partial_Lv(L,\cdot)$ admit uniform clock behavior on $I$, i.e., for fixed $j\in\bbZ$ 
	\begin{align}\label{spacing}
	\lim\limits_{L\to\infty}Lf_\E(\xi)\left(\xi_{j+1}^{L}(\xi)-\xi_{j}^{L}(\xi)\right)=1
	\end{align} 
	uniformly for $\xi\in I$. 
\end{theorem}

The organization of the paper is as follows.  In Section 2 we recall concepts from the theory of Stahl--Totik regularity for continuum Schr\"odinger operators and prove that the additional assumption of Dirichlet regularity of $\E$ leads to uniformity in the asymptotic estimates. This is crucial to control the exponential growth of the Neumann solution close to $\E$. In Section 3 we recall aspects of the spectral theory for continuum Schr\"odinger operators and show how this can be viewed as a special case of the general theory of canonical systems. In particular, we show that the Christoffel function can also be defined through an extremal problem. In Section 4 we prove asymptotics of the Christoffel function for finite gap potentials. Section 5 is devoted to the proofs of the main theorems. We also provide an appendix in which we recall some parts from the theory of canonical systems. 

\subsection*{Acknowledgements}
I would like to thank Milivoje Luki\'c and Brian Simanek for helpful discussions. 

\section{Stahl-Totik regularity}

From the discrete setting it is known that in order to show \eqref{CDPoly}, in addition to local properties of the measure, some additional global assumption is needed \cite{Totik2000AssympChris}. A common sufficient assumption is to assume that the measure is regular in the sense of Stahl and Totik. In \cite{EichLuk} a corresponding theory was developed for continuum Schr\"odinger operators with  uniformly locally integrable potential.  In the following, we show that an additional uniformity is obtained in the estimates under the additional assumption that the underlying spectrum is regular for the Dirichlet problem. 

For $z\in\bbC$ the Dirichlet solution, $u(x,z)$ is the solution of \eqref{defnv} with initial condition $u(0,z)=0, \partial_xu(0,z)=1$. Stahl--Totik regularity as defined in the introduction was linked to exponential growth of the Dirichlet solutions. However, the same proofs also characterize the exponential growth of the Neumann solution. By $\cD'(\bbC)$ we denote the space of distributions and recall that subharmonic functions can be viewed as representatives of real-valued distributions with non-negative Laplacian. 

\begin{lemma}\label{lem:NeumannRegular}
	Let $V$ obey \eqref{L1locunif} and $v(x,z)$ denote the Neumann solution.
	Then:
	\begin{enumerate}[(a)]
		\item\label{it:Reg1} For any $x > 0$,  the function $\log \lvert v(x,z) \rvert$ is a subharmonic function on $\bbC$.
		\item\label{it:Reg2} The family of functions $\{\log \lvert v(x,z) \rvert \}_{x \ge 1}$ is precompact in $\cD'(\bbC)$.
	\end{enumerate}
Moreover, if $V$ is Stahl--Totik regular, then:
\begin{enumerate}[(i)]
		\item\label{it:Reg3} The functions  $\log \lvert v(x,z) \rvert$ converge as $x \to \infty$ to $M_\E(z)$ in the distributional sense as well as uniformly on compact subsets of $\bbC_+$. 
		\item\label{it:Reg4} For  all $z\in\bbC$, $\limsup_{x\to\infty}\frac{1}{x}\log|v(x,z)|\leq M_\E(z)$. 
	\end{enumerate}
\end{lemma}

\begin{proof}
	\eqref{it:Reg1}: By general principles, for any $x > 0$, $v(x,z)$ is an entire function of $z$ \cite{PoschelTrubowitz}, so $\log \lvert v(x,z)\rvert$ is subharmonic.
	
	\eqref{it:Reg2}: Is analogous to the proof of \cite[Theorem~4.3]{EichLuk}: it is a consequence of locally uniform upper bounds which follow from general principles, and a pointwise lower bound which follows from boundedness of the diagonal Green function for the Schr\"odinger operator with Neumann boundary conditions.

	\eqref{it:Reg3}: If $V$ is Stahl--Totik regular and $u(x,z)$ denotes the Dirichlet solution, then by \cite[Theorem 1.5]{EichLuk} for any $z \in \bbC \setminus \bbR$,
	\[
	\lim_{x\to\infty} \frac 1x \log \lvert u(x,z) \rvert = M_\E(z).
	\]
	In words, $u(x,z)$ grows exponentially with rate $M_\E(z) > 0$. Moreover, the Weyl solution $\psi(x,z)$ decays, $\lim_{x\to\infty} \psi(x,z) = 0$. The Neumann solution $v(x,z)$ is a linear combination of $u(x,z)$ and $\psi(x,z)$ and it is not a multiple of $\psi(x,z)$ (since $z$ is not real, it is not an eigenvalue of the self-adjoint operator). Thus, $v(x,z)$ also obeys
	\[
	\lim_{x\to\infty} \frac 1x \log \lvert v(x,z) \rvert = M_\E(z), \qquad \forall z \in \bbC \setminus \bbR.
	\]
	This implies that $M_\E$ is the only possible subsequential limit of $\frac 1x \log \lvert v(x,z) \rvert$ in $\cD'(\bbC)$. By precompactness of the family, this implies convergence to $M_\E$ in the topology of $\cD'(\bbC)$.
	
	\eqref{it:Reg4}: Follows from convergence in $\cD'$ together with the principle of descents for subharmonic functions. 
\end{proof}

We will show that for Dirichlet regular sets \eqref{it:Reg4} holds uniformly on compact subsets of $\bbC$. This will follow from showing that \eqref{it:Reg4} holds not only pointwise but in the following stronger sense: Given a sequence $(z_n)_{n\in\bbN}$ such that $\lim_{n\to\infty}z_n=z_\infty\in\bbC$ and and increasing sequence $(x_n)_{n\in\bbN}$, $\lim_{n\to\infty}x_n=\infty$, then 
\[
\limsup_{n\to\infty}\frac{1}{x_n}\log|u(x_n,z_n)|\leq M_\E(z_\infty).
\]
If then in addition $M_\E$ is continuous (or merely lower semicontinuous), this implies that \eqref{it:Reg4} holds uniformly on compact subsets of $\bbC$. Since for Dirichlet regular sets $\E$, $M_\E$ is continuous \cite[Theorem 2.1]{EichLuk} this will prove the above claim.

We will need a version of principle of descents for Green potentials. Let $\Omega\subset\bbC$ be a Greenian domain and denote its Green function by $G_\Omega(z,w)$, cf. \cite[Section 4]{Classpotential}.  Let $\nu$ be a measure supported in $\Omega$ and define the Green potential of $\nu$ by 
\[
\Phi^G_\nu(z)=\int G_\Omega(z,t)d\nu(t)
\]
provided that there exists $z_0\in\Omega$ such that $\Phi^G_\nu(z_0)<\infty$. It defines a superharmonic function in $\Omega$. By \cite[Lemma 4.2.2]{Classpotential} this holds in particular, if $\nu$ is supported on a compact subset of $\Omega$. 
 
\begin{lemma}\label{lem:POD}
	Let $\Omega$ be a Greenian domain and let $\nu_n,\nu_\infty$ be finite measures with support in a common compact subset of $\Omega$ and $\lim_{n\to\infty}\nu_n=\nu_\infty$ in the weak-$*$ sense. Let $z_n\in\Omega$ with $\lim_{n\to\infty}z_n=z_\infty\in \Omega$. Then 
	\[
	\liminf_{n\to\infty}\Phi^G_{\nu_n}(z_n)\geq \Phi^G_{\nu_\infty}(z_\infty).
	\]
\end{lemma}
\begin{proof}
	By assumption there exists a compact set $K\subset\Omega$ such that $\supp\nu_n\cup\supp\nu_\infty\cup\{z_n\}_{n\in\bbN}\cup\{z_\infty\}\subset K$.
	For $M>0$ define
	\[
	G^M_\Omega(z,t)=\min\{ M, G_\Omega(z,t)\}.
	\]
	The claim will follow from uniform continuity of $G^M_\Omega$ on $K\times K$. By \cite[Theorem 4.1.9.]{Classpotential} $G_\Omega(z,t)$ is continuous on $K\times K$ in the extended sense (i.e. with the value $+\infty$ allowed). Thus it follows that $G^M_\Omega(z,t)$ is continuous on $K\times K$ and $G^M_\Omega(z,t)\leq M$. We conclude that $G^M_\Omega(z,t)$ is continuous in the standard sense and since $K$ is compact uniform continuity follows. 
	Since $\nu_\infty$ is a finite measure and $\nu_n\to\nu_\infty$ we have $\sup \nu_n(K)<\infty$. Now it follows from the monotone convergence theorem that 
	\begin{align*}
		\Phi^G_{\nu_\infty}(z_\infty)&=\lim\limits_{M\to\infty}\int G^M_\Omega(z_\infty,t)d\nu_\infty(t)\\
		&=\lim\limits_{M\to\infty}\lim_{n\to\infty}\int G^M_\Omega(z_n,t)d\nu_n(t)\\
		&\leq\lim\limits_{M\to\infty}\liminf_{n\to\infty}\int G_\Omega(z_n,t)d\nu_n(t)\\
		&=\liminf_{n\to\infty}\Phi^G_{\nu_n}(z_n).
	\end{align*}
\end{proof}

Following the proof of \cite[Theorem 2.7.4.1]{Azarin09} we obtain:
\begin{theorem}\label{thm:upperEnvSuperharmonic}
	Let $(u_n)_{n\in\bbN}$ and $u_\infty$ be superharmonic functions in $\bbC$ such that $\lim_{n\to\infty}u_n=u_\infty$ in $\cD'(\bbC)$. Let $z_n\in\bbC$ with $\lim_{n\to\infty}z_n=z_\infty\in \bbC$. Then
	\[
	\liminf_{n\to\infty}u_n(z_n)\geq u_\infty(z_\infty).
	\]
\end{theorem}
\begin{proof}
	Fix $R>0$ such that $z_n,z_\infty\in B_{R}(0)$. By \cite[Theorem 2.7.1.1]{Azarin09} the $u_n$ are uniformly bounded from below on $B_{2R}(0)$ and thus we can assume that the $u_n$ and $u_\infty$ are non-negative there. 
	
	For a non-negative superharmonic function $u$ in a domain $\Omega$ and $E\subset\Omega$ let $\hat R_u^E$ denote the regularized reduced function \cite[Section 5.3]{Classpotential}\footnote{In \cite[Theorem 2.7.2.1]{Azarin09} the regularized reduction was introduced in a different way, but it follows from its characteristic properties that these two notions, up to switching from super- to subharmonic functions, coincide.}. View $u_n, u$ as non-negative superharmonic functions in $B_{2R}(0)$.  Set $K=\overline{B_{R}(0)}$ and
	\[
	v_n=\hat R_{u_n}^{K},\quad v=\hat R^{K}_u.
	\]
	Since $ K\subset B_{2R}(0)$ is compact we see that $v_n$ and $v$ are Green potentials \cite[Theorem 5.3.5]{Classpotential}, i.e., $v_n=\Phi^G_{\nu_n}$, $v=\Phi^G_{\nu}$ for some measures $\nu_n,\nu$ and the Green function of the domain $B_{2R}(0)$. By \cite[Theorem 5.3.4]{Classpotential} $v_n, v$ are harmonic outside of $K$ and it follows that $\nu_n,\nu$ are supported on $K$. Due to \cite[Theorem 2.7.2.2]{Azarin09}, $v_n\to v$ in $\cD'(B_{2R}(0))$ and thus since by \cite[Theorem 4.3.8]{Classpotential} $\nu_n,\nu$ are the Riesz measures of $v_n$ and $v$ we conclude that $\nu_n\to\nu$ in $C(K)^*$. Therefore, we obtain from Lemma \ref{lem:POD} that 
	\[
	\liminf_{n\to\infty}\Phi^G_{\nu_n}(z_n)\geq \Phi^G_{\nu}(z).
	\]
	Since $v_n=u_n$ and $v=u$ on $B_R(0)$ the claim follows. 
\end{proof}

We immediately get the following:
\begin{theorem}\label{thm:locUnif}
	If $V$ is Stahl--Totik regular then for any $z_n\to z$ and $x_n\to\infty$ we have 
	\begin{align}\label{eq:locUnif}
	\limsup_{n\to\infty}\frac{1}{x_n}\log|v(x_n,z_n)|\leq M_\E(z).
	\end{align}
	In particular, if $\E$ is Dirichlet regular, uniformly on compact subsets of $\bbC$ we have 
	\[
	\limsup_{x\to\infty}\frac{1}{x}\log|v(x,z)|\leq M_\E(z).
	\]
\end{theorem}
\begin{proof}
	As in the proof of Lemma
	 \ref{lem:NeumannRegular} regularity implies that for any $x_n\to\infty$  the family of subharmonic functions $(\frac{1}{x_n}\log|v(x_n,z_n)|)$ converges to $M_\E(z)$ in $\cD'(\bbC)$ and is bounded from above on compact subsets of $\bbC$. Thus, the first claim follows from Theorem \ref{thm:upperEnvSuperharmonic} applied to $(-\frac{1}{x_n}\log|v(x_n,z_n)|)$. If $\E$ is Dirichlet regular, $M_\E$ is continuous on $\bbC$. The second statement then follows  from \eqref{eq:locUnif} and lower semi-continuity of $M_\E$. 
\end{proof}

We need to control the Neumann solution for large real energies. We only need bounds for real spectral parameter $\xi$, but $\xi$ can be negative. We will always assume that $ \Im \sqrt z\geq 0$ for $z\in \bbC$. The Neumann and Dirichlet solutions for $V=0$ are the functions
\[
c(x,z) = \cos(\sqrt{z}x), \qquad s(x,z) = \begin{cases} \frac{ \sin(\sqrt{z}x)}{\sqrt{z}} & z \neq 0 \\ x & z = 0 \end{cases}.
\]
By standard arguments,  for general $V\in L^1([0,x])$, the initial value problem \eqref{defnv} is rewritten as integral equations, and by Volterra-type arguments, convergent series representations are then found for the fundamental solutions. With the notation
$\Delta_n(x) = \{ t \in \R^n \mid x \ge t_1 \ge t_2 \ge \dots \ge t_n \ge 0 \}$, the series representation for the Neumann solution is 
\begin{align}
v(x,z) & = c(x,z) + \sum_{n=1}^\infty \int_{\Delta_n(x)} s(x-t_1,z) \left( \prod_{j=1}^{n-1} V(t_j) s(t_{j}-t_{j+1},z) \right) V(t_n)  c(t_n,z) \dd^n t \label{fundamentalseries2},
\end{align}
see \cite{PoschelTrubowitz} or \cite[Section 3]{EichLuk}.
\begin{lemma}
	Let $x>0$. For $\xi\geq 1$ we have 
	\begin{align}\label{eq:EstimatesSolution1}
	|v(x,\xi)|&\leq e^{\frac{\int_{0}^{x}|V(t)|dt}{\sqrt{\xi}}}.
	\end{align}
	For $\xi<1$ we have 
	\begin{align}\label{eq:EstimatesSolution2}
	|v(x,\xi)|\leq e^{(1+\Im \sqrt{\xi})x+\int_{0}^{x}|V(t)|dt}.
	\end{align}
\end{lemma}
\begin{proof}
	For $\xi\geq 1$ we use the estimates $|c(x,\xi)|\leq 1$ and $|s(x,\xi)|\leq \sqrt{\xi}^{-1}$ to get from \eqref{fundamentalseries2}
	\[
	|v(x,\xi)|\leq 1+\sum_{n=1}^\infty \frac{\int_{\Delta_n(x)} \prod_{j=1}^{n} |V(t_j) |\dd^n t}{\sqrt{\xi}^{n}}=1+\sum_{n=1}^\infty\frac{\left(\int_0^x|V(t)|dt\right)^n}{\sqrt{\xi}^{n}n!} =e^{\frac{\int_{0}^{x}|V(t)|dt}{\sqrt{\xi}}}
	\]
	Similarly, for $\xi<1$ we use  $\lvert s(x,\xi) \rvert = \lvert \int_0^x c(t,\xi) \dd t \rvert \le x e^{\Im \sqrt{\xi} x} \le e^{(1+\Im \sqrt{\xi})x}$ and $\lvert c(x,\xi) \rvert\leq e^{\Im \sqrt{\xi}x} \leq e^{(1+\Im \sqrt{\xi})x}$ and get 
	\begin{align*}
		|v(x,\xi)|\leq e^{(1+\Im \sqrt{\xi})x}e^{\int_{0}^{x}|V(t)|dt}.
	\end{align*}
\end{proof}

Combining these estimates for large real energies and Theorem \ref{thm:locUnif} on compact subsets of $\bbC$, we get an uniform growth estimate for the Neumann solutions for  $\xi\in\bbR$ which are close to $\E$. To this end let us introduce for $0<\d<1$ the extension
\begin{align}\label{eq:dextension}
\E_\d=\{\xi\in\bbR:\ \dist(\xi,\E)<\d\}\cup[1/\xi,\infty).
\end{align}
Having in mind that $\infty$ is a boundary point of $\Omega$ it is natural to also add a half line. Note that this makes $\E_\d$ a finite gap set, which will be crucial in Section \ref{sec:Variational}.
\begin{theorem}\label{thm:regBounds}
	Let $V$ be Stahl--Totik regular such that $\E=\sigma_{\ess}(V)$ is Dirichlet regular. Then for any $\e>0$ sufficiently small there exists $0<\d<1$ and $C>0$ such that for any $\xi\in \E_\d$ and $x>0$ we have
	\[
	|v(x,\xi)|\leq Ce^{\e x}.
	\]
\end{theorem}
\begin{proof}
	Let  $c_1 =\sup_{x\geq 1}\int_x^{x+1}|V(t)|dt$ and fix $0<\e<2c_1$. We have $\int_0^x|V(t)|dt\leq c_1(x+1) \leq 2c_1x$ for $x \geq  1$. Setting $\d_1=\left(\frac{\e}{2c_1}\right)^2$ we obtain the estimate for $\xi\geq 1/\d_1$ from  \eqref{eq:EstimatesSolution1}. 
	
	Let $K=[\inf\E -1,1/\d_1]$. Since $M_\E$ is uniformly continuous on $K$ and vanishes on $\E$, we find $\d_2>0$ such that for $\xi\in K$ with  $\dist(\xi,\E)\leq\d_2$, $M_\E(\xi)<\frac{\e}{2}$. Thus, by Theorem \ref{thm:locUnif} there exists $L_0>0$ such that for $x\geq L_0$ and $\xi\in K$ with $\dist(\xi,\E)<\d_2$ 
	\[
	\frac{1}{x}\log|v(x,\xi)|\leq \e. 
	\] 
	For $x\leq L_0$ we use \eqref{eq:EstimatesSolution2} to get a uniform bound for $|v(x,\xi)|$. The claim follows with $\d=\min\{\d_1,\d_2\}$. 
\end{proof}
\section{Christoffel Darboux kernel, Christoffel function and Fourier transform}\label{sec:deBrangesSchroedinger}
In this section we recall aspects of the spectral theory of continuum Schr\"odinger operators on the half-line. By relating it to the rich theory of canonical systems developed by Krein and de Branges, we will characterize the Christoffel-Darboux kernel as a reproducing kernel of a certain Hilbert space of entire functions. This viewpoint will also give an expression of $\l_L(z)$ in terms of an extremal problem. The connection between canonical systems and continuum Schr\"odinger operators has been extensively discussed by Remling in \cite{RemlingJFA02,RemlingAfunction}.

For a given potential $V:[0,\infty)\to \bbR$, $V\in L^1_\loc$, let $H_V$ be the associated Schr\"odinger operator on the half-line, where we again assume a  Neumann boundary condition at $0$ and that $V$ is so that  $\infty$ is a limit point endpoint. If we consider the truncation to $[0,L]$ then we assume in addition a Neumann boundary condition at $L$
\[
f'(L)=0
\]
and denote the corresponding operator by $H_V^L$.

Let us start by considering $H_V^L$. It is well known that the spectrum is purely discrete and simple. Let $v(x,z)$ and $u(x,z)$ denote the Neumann and Dirichlet solution, respectively. 
Define the measure
\[
\mu_L=\sum_{\xi:\ \partial_Lv(L,\xi)=0}\frac{\delta_\xi}{\|v(\cdot,\xi)\|_{L^2_{dx}((0,L))}},
\]
where $\d_\xi$ denotes the Dirac measure. Note that this is not the same measure as $\nu_L$ defined in the introduction. The operator $U:L^2_{dx}((0,L))\to L^2_{d\mu_L}(\bbR)$
\begin{align}\label{eq:eigenvalueExp}
(Uf)(z)=\int_0^L f(x)v(x,z)dx
\end{align}
is unitary and
\[
UH_V^LU^*=S_z,
\]
where $S_z$ denotes the operator of multiplication with the independent variable in $L^2_{d\mu_L}(\bbR)$.  The adjoint of $U$ is given by 
\[
(U^*F)(x)=\int v(x,\xi)F(\xi)d\mu_L(\xi).
\]

Although in the above discussion $U(L^2_{dx}((0,L)))$ was considered as $L^2_{d\mu_L}(\bbR)$, \eqref{eq:eigenvalueExp} allows to interpret for $f\in L^2_{dx}((0,L))$, $F(z)=(Uf)(z)$ as a function on $\bbC$. This leads to the theory of de Branges spaces. 
Define the transfer matrix 
\[
T(x,z)=\begin{pmatrix}
v(x,z)& -u(x,z)\\-\partial_xv(x,z)& \partial_xu(x,z)
\end{pmatrix}
\]
and note that it solves the differential equation 
\begin{align}\label{eq:SchroedingerCanonical}
j\partial_xT(x,z)=\left(-z\begin{pmatrix}1&0\\0&0\end{pmatrix}+\begin{pmatrix}
V(x)&0\\0&-1
\end{pmatrix}\right)T(x,z),\quad j=\begin{pmatrix}
0&-1\\1&0
\end{pmatrix}.
\end{align}
Such a system is called a canonical system with coefficient functions $A(x)=\begin{pmatrix}1&0\\0&0\end{pmatrix}$ and $B(x)=\begin{pmatrix}
V(x)&0\\0&-1
\end{pmatrix}$; see Appendix \ref{app1}. It follows from \eqref{eq:HB} that 
\[
E_L(z)=v(L,z)+i\partial_Lv(L,z)
\]
is a Hermite-Biehler function. This means that $E$ is entire, has no zeros in $\bbC_+$ and satisfies $|E(z)|\geq |E(\overline z)|$ there. To a Hermite-Biehler function we can associate a Hilbert space of entire functions $B(E)$ with scalar product
\begin{align}\label{eq:deBrangesNorm}
\langle F,G\rangle_{B(E)}=\frac{1}{\pi}\int F(x)\overline{G(x)}\frac{dx}{|E(x)|^2};
\end{align}
see Appendix \ref{app1}. 
\begin{theorem}{\cite[Theorem 3.1]{RemlingJFA02}}
	The Hilbert space $B(E_L)$ and $L^2_{d\mu_L}(\bbR)$  are identical. More precisely, if $F\in B(E_L)$, then the restriciton of $F$ to $\bbR$ belongs to $L^2_{d\mu_L}(\bbR)$ and $F\mapsto F|_\bbR$ is unitary. 
\end{theorem}

For an entire function $F$ we denote $F^\#(z)=\overline{F(\overline{z})}$. The Hilbert spaces $B(E_L)$ are reproducing kernel Hilbert spaces and the kernel is given by
\begin{align}\label{def:CDKernel}
	K_L(z,w)=\frac{\overline{E_L(w)}E_L(z)-\overline{E_L^\#(w)}E_L^\#(z)}{2i(\overline{w}-z)}=\frac{\overline{v(L,w)}v'(L,z)-\overline{v'(L,w)}v(L,z)}{\overline{w}-z}.
\end{align}
On the other hand, using \eqref{eq:CDFormula}, it follows that 
\begin{align}\label{eq:CDKernel2}
K_L(z,w)=\int_0^Lv(x,z)\overline{v(x,w)}dx.
\end{align} 
Using the terminology common in the orthognal polynomials literature, we call $K_L(z,w)$ the Christoffel-Darboux kernel. In the setting of orthogonal polynomials, the equivalence of \eqref{def:CDKernel} and \eqref{eq:CDKernel2} is called the Christoffel-Darboux formula.
Evaluating \eqref{def:CDKernel} on the diagonal for $\xi\in\bbR$ gives
\begin{align}\label{eq:Diagonal}
	K_L(\xi,\xi)=(\partial_L v)(L,\xi)(\partial_\xi v)(L,\xi)-v(L,\xi)(\partial_\xi\partial_Lv)(L,\xi).
\end{align}

It is a remarkable property that as sets $B(E_L)$ do not depend on the potential, see \cite[Theorem 4.1]{RemlingJFA02}. That is, for any $V\in L^1([0,L]),$
\begin{align}\label{eq:SetEqual}
B(E_L)=S_L:=\left\{\int_0^Lf(x)\cos(\sqrt{z}x)dx:\ f\in L^2((0,L)) \right\}
\end{align}
where this is understood as set equality. Of course the topology depends on the potential through \eqref{eq:deBrangesNorm}.
Recall that $c(x,z)=\cos(\sqrt{z}x)$ is the Neumann solution for $V=0$. Using the variable $k^2=z$ it will be convenient to also consider the set
\begin{align}\label{eq:cSL}
\fS_L=\left\{\int_0^Lf(x)\cos(kx)dx:\ f\in L^2((0,L))\right \}.
\end{align}
It is mentioned in \cite[page 101]{MarchenkoSturmLiouville} and it follows from the Paley-Wiener theorem that
\begin{align*}
\fS_L=\left\{g\in L_{dk}^2(\bbR): g\text{ is entire, even and of exponential type at most }L\right\}.
\end{align*}
From \eqref{eq:SetEqual} it follows directly that $B(E_{L_1})\subset B(E_{L_2})$ for $L_1\leq L_2$ and in fact this inclusion is isometric. But even more is true:
\begin{theorem}{\cite[Theorem 3.2]{RemlingJFA02}}
	Suppose $0<L_1\leq L_2$. Then $B(E_{L_1})$ is isometrically contained in $B(E_{L_2})$. Moreover, if $\mu$ denotes the spectral measure for the half-line problem, then for every $L>0$, $B(E_{L})$ is isometrically contained in $L^2_{d\mu}(\bbR)$ in the sense that for $F\in B(E_{L})$, $F|_{\bbR}\in L^2_{d\mu}$ and $\|F\|_{B(E_{L})}=\|F|_{\bbR}\|_{L^2_{d\mu}}$. 
\end{theorem}

In the following, we will drop the restriction map and interpret $F$ either as function in $L^2_{d\mu}$ or as an entire function in $B(E_L)$ depending on the context. This shouldn't lead to any confusion. 

Using that $B(E_L)$ is isometrically contained in $L^2_{d\mu}(\bbR)$  and the equality \eqref{eq:SetEqual} we can define the Christoffel function:
\begin{definition}
	For $L>0$ and $z\in\bbC$ we define the Christoffel function associated to $V$ by 
	\[
	\l_L(z)=\inf\{\|F\|^2_{L^2_{d\mu}}: F\in S_L, F(z)=1\}.
	\]
\end{definition}
Since $B(E_L)$ is a reproducing kernel Hilbert space it follows from the Cauchy--Schwarz inequality that the infimum is in fact a minimum and that the extremizer is given by 
\[
Q_{L}(z,z_0)=\frac{K_L(z,z_0)}{K_L(z_0,z_0)}.
\]
Note that $K_L(z_0,z_0)=\|K_L(\cdot, z_0)\|^2_{B(E_L)}=\|K_L(\cdot, z_0)\|^2_{L^2_{d\mu}}$. In particular,
\[
\l_L(z)=\frac{1}{K_L(z,z)}.
\]
Combining this with \eqref{eq:CDKernel2} shows that this definition of $\l_L(z)$ coincides with the one given in \eqref{def:Christoffel}.
\section{Asymptotics for finite gap potentials}\label{sec:FiniteGap}
The goal of this section is to show \eqref{eq:CFasymp} for so-called finite gap potentials. We will need some preliminary observations:

\subsection{The isospectral torus}
In this section we assume that $\E$ is a finite gap set of the form 
\[
\E=[b_0,\infty)\setminus\bigcup_{j=1}^g(a_j,b_j),
\]
where $b_0<a_j<b_j<a_{j+1}<b_{j+1}$, for $0<j<g-1$. In this setting it is more natural to consider Schr\"odinger operators on $\bbR$. Let $V$ be a continuous and bounded potential on $\bbR$ and 
$
H_V
$ 
acting on its natural domain in $L^2(\bbR)$. Clearly $H_{V_\pm}$ with potential $V_+=V|_{[0,\infty)}$ and $V_-: [0,\infty)\to \bbR$ defined by $V_-(x):=V(-x)$ have $\infty$ as an limit point endpoint and thus we can associate Weyl $m$-functions $m_{\pm}$ by \eqref{eq:mfunction}.

\begin{definition}
	We say that $V$ is reflectionless on  $A\subset\bbR$ if for almost every $\xi\in A$
	\[
	m_+(\xi+i0)=- m_-(\xi-i0).
	\]
	For a given finite gap set $\E$ we define the isospectral torus by 
	\[
	\cT(\E)=\{V\in C_b(\bbR):\ \sigma(H_V)=\E\text{ and } V \text{ is reflectionless on }\E \}.
	\]
\end{definition}
This class has been considered for essentially more general sets \cite{GesYud,SodinYudSturmLiouville}. 

\subsection{Abelian integrals}
For finite gap potentials many important spectral theoretical objects can be given explicitly in terms of the Abelian integrals on the associated hyperelliptic Riemann surface $\cR_\E$ associated to $\sqrt{(z-b_0)\prod(z-a_j)(z-b_j)}$, i.e., 
\[
\cR_\E=\{(z,w)\in \bbC^2:\ w^2=(z-b_0)\prod_{j=1}^g(z-a_j)(z-b_j)\}\cup\{\infty\}.
\]
Typically $\cR_\E$ is visualized as two copies of $\bbC\setminus \E$, corresponding to the two branches of the square root, glued together along $\E$. For more details see \cite{GeHoSoliton}.  These representations and the properties which follow from them allow us to compute the limit of the Christoffel functions explicitly.

For finite gap sets the Martin function can be given in terms of a Schwarz--Christoffel mapping. Define
\begin{align*}
\t_\E(z)=\int_{b_0}^z\frac{-1}{2\sqrt{u-b_0}}\prod_{j=1}^g\frac{(u-c_j)}{\sqrt{(u-a_j)(u-b_j)}}du,
\end{align*}
where $c_j\in(a_j,b_j)$ is uniquely determined by
\begin{align*}
\t_\E(b_j)=\t_\E(a_j).
\end{align*}
The function $\t_\E$ is a conformal mapping of $\bbC_+$ to a comb 
\[
\Pi_\E=\{\xi+iy:\ \xi,y>0\}\setminus\bigcup_{j=1}^g\{\eta_j+iy: 0<y<h_j\},
\]
where $\eta_j$ are called frequencies and $h_j$ heights. Note that $\t_\E$ can be extended to $\bbR$,
\[
\t_\E(z)=\sqrt{z},\quad \text{ as } z\to -\infty,\quad \t_\E(b_0)=0
\]
and that 
$\t_\E^{-1}(\bbR_+)=\E$. It follows from these properties that $M_\E(z)=\Im \t_\E(z)$ is the Martin function of the domain.  These type of comb mappings are commonly used in inverse spectral theory and in uniform approximation problems \cite{AkhLev60,MarOst75} or \cite{ErY12} for a modern approach to the subject. Since $M_\E$ can be extended to a subharmonic function on $\bbC$, its distributional Laplacian is a positive measure and we can define $\rho_\E=\frac{1}{2\pi}\Delta M_\E$.  The Riesz representation then yields
\begin{align*}
M_\E(z)=M_\E(z_*)+\int_\E\log\left|1-\frac{z-z_*}{t-z_*}\right|d\rho_\E,
\end{align*}
where $z_*<b_0$ is some normalization point. Computing $\frac{1}{2\pi}\Delta M_\E$ we find that 
\begin{align}\label{eq:rhoPrime}
d\rho_\E(\xi)=\frac{1}{\pi}\t_\E'(\xi)d\xi=\frac{-1}{2\pi\sqrt{\xi-b_0}}\prod_{j=1}^g\frac{(\xi-c_j)}{\sqrt{(\xi-a_j)(\xi-b_j)}}d\xi.
\end{align}
That is, $\rho_\E$ is purely a.c. and its density is given by
\begin{align*}
\frac{d\rho_\E(\xi)}{d\xi}=f_\E(\xi)=\frac{-1}{2\pi\sqrt{\xi-b_0}}\prod_{j=1}^g\frac{(\xi-c_j)}{\sqrt{(\xi-a_j)(\xi-b_j)}}.
\end{align*}
In particular, we see that $f_\E(\xi)$ is real analytic inside any band of $\E$, see also Lemma \ref{lem:AppendixMartinMeasure}. We will also constantly use that $\t_\E(\xi)\in\R$ for $\xi\in\E$.

Let $\psi_+$ denote the Weyl solution at $+\infty$, cf. \eqref{eq:WeylSol}. Due to \cite[p. 462]{GeHoSoliton}, $\psi_+$ corresponds to the restriction of the Baker-Akhiezer function to the upper sheet and thus by \cite[Theorem 1.20]{GeHoSoliton} it can be represented as 
\begin{align*}
\psi_{+}(x,z)=e^{i\t_\E(z)x}f(x,z),
\end{align*}
where $f$ is given in terms  of Theta functions on $\cR_\E$. In the following let $[a,b]=I\subset \Int(\E)$.   We can extend $\psi_{+}(x,z)$ analytically to $I$. Moreover, we have
\begin{enumerate}[(i)]
	\item $x\mapsto f(x,\xi)$ is almost periodic\footnote{this follows from continuity of the Theta function and the linearization of the Abel map in \cite[Theorem 1.20]{GeHoSoliton}},
	\item $\xi\mapsto f(x,\xi)$ is analytic on $I$ and all derivatives are uniformly bounded for $x>0$ and $\xi\in I$.
\end{enumerate}

\subsection{Asymptotics of $\l_L(\xi)$}
In the following let $[a,b]=I\subset \Int(\E)$. If we write $f(\xi)$ or $m_+(\xi)$, $\xi\in I$, we mean the corresponding limits $\xi+i\e$ as $\e\to 0$. All of them can be analytically extended to $I$. Moreover, we have $\Im  m_+>0$ there. To avoid confusion, we mention that for $ m_+$, this does not correspond to the extension $\overline{ m(\overline z)}$, which is an extension through $\bbR\setminus\E$. 

Let 
\begin{align*}
W(f,g)(x)=f(x)g'(x)-f'(x)g(x)
\end{align*}
denote the Wronskian of $f$ and $g$. For $\xi\in I$, $\psi_+(x,\xi)$ and $\overline{\psi_+(x,\xi)}$ both solve \eqref{eq:WeylSol} and thus their Wronskian is constant. For $\xi\in I$, $\partial_x\psi_+(0,\xi)\neq 0$, since this would lead to an eigenvalue and $\sigma(H_V)$ is purely absolutely continuous there. Recall that $m_+$ is given by
\[
m_+(z)=-\frac{\psi_+(0,z)}{\partial_x\psi_+(0,z)}.
\]
Thus we see that by constancy of the Wronskian
\begin{equation}\label{Wronsk1}
\begin{aligned}
W_\xi(\psi_+,\overline{\psi_+}):=W(\psi_+(\cdot,\xi),\overline{\psi_+(\cdot,\xi)})(x)&=W(\psi_+(\cdot,\xi),\overline{\psi_+(\cdot,\xi)})(0)\\
&=-2i|\partial_x\psi_+(0,\xi)|^2\Im m_+(\xi)\neq 0.
\end{aligned}
\end{equation}
It follows now by direct verification of \eqref{defnv}  that 
\[
v(x,\xi)=\frac{\psi_{+}(x,\xi)\overline{\partial_x\psi_+(0,\xi)}-\overline{\psi_{+}(x,\xi)}\partial_x\psi_+(0,\xi)}{W_\xi(\psi_{+},\overline{\psi_{+}})}
\]
defines he Neumann solution for $H_V$. 

Let us set
\begin{align*}
c(\xi)=\frac{\overline{\partial_x\psi_{+}(0,\xi)}}{W_\xi(\psi_{+},\overline{ \psi_{+}})}
\end{align*}
and
\[
h(x,\xi)=c(\xi)\psi_+(x,\xi).
\]
It will also be convenient to take off the exponential part and consider
\[
g(x,\xi)=e^{-i\t_\E(\xi)x}h(x,\xi).
\]
Using $\overline{W_\xi(\psi_{+},\overline{ \psi_{+}})}=-W_\xi(\psi_{+},\overline{ \psi_{+}})$ and $\t_\E(x)\in\bbR$, we see that
\begin{align*}
v(x,\xi)=h(x,\xi)+\overline{h(x,\xi)}=e^{i\t_\E(\xi)x}g(x,\xi)+e^{-i\t_\E(\xi)x}\overline{g(x,\xi)}.
\end{align*}
The Herglotz function $m_+$ can be continuously extended to $I$. Thus, if $d\mu(\xi)=f_\mu(\xi)d\xi+d\mu_s(\xi)$ denotes the spectral measure of $H_V$, we have
\[
f_\mu(\xi)=\frac{1}{\pi}\Im m_+(\xi).
\]

\begin{lemma}\label{lem:WronskianIdent}
	For $\xi\in I$, we have   
	\begin{align*}
	W(h(\cdot,\xi),\overline{h(\cdot,\xi)})(x)=-2i\t_\E(\xi)|g(x,\xi)|^2+W(g(\cdot,\xi),\overline{g(\cdot,\xi)})(x)=\frac{1}{2\pi if_\mu(\xi)}.
	\end{align*}
\end{lemma} 
\begin{proof}
	The first equality follows by direct computation. Using the definition of $h$ and \eqref{Wronsk1}, we see that 
	\begin{align*}
	W(h(\cdot,\xi),\overline{h(\cdot,\xi)})(x)&=W(c(\xi)\psi_{+}(\cdot,\xi),\overline{c(\xi)\psi_{+}(\cdot,\xi)})(x)\\
	&=|c(\xi)|^2W_\xi(\psi_{+}(\cdot,\xi),\overline{\psi_{+}}(\cdot,\xi))(x)\\
	&=\frac{|\partial_x\psi_+(0,\xi)|^2}{|W_\xi(\psi_+,\overline{ \psi_+})|^2 }W_\xi(\psi_{+}(\cdot,\xi),\overline{\psi_{+}}(\cdot,\xi))(x)=\frac{1}{2\pi if_\mu(\xi)}.
	\end{align*}
\end{proof}

We are now ready to prove the main theorem of this section 
\begin{theorem}
	Let $\xi\in I=[a,b] \subset \Int(\E)$. Then
	\begin{align*}
	K_x(\xi,\xi)=x\frac{f_\E(\xi)}{f_\mu(\xi)}+O(1)
	\end{align*}
	as $x\to\infty$,where the $O(1)$ is uniform for $\xi\in I$.
\end{theorem}
\begin{proof}
	We use the representation
	\[
	v(x,\xi)=e^{i\t_\E(\xi)x}g(x,\xi)+e^{-i\t_\E(\xi)x}\overline{g(x,\xi)}.
	\]
	All $O$ notation is related to $x\to\infty$. We have 
	\begin{align*}
	\partial_\xi v(x,\xi)&=i\t'_\E(\xi)x\left(e^{i\t_\E(\xi)x}g(x,\xi)+e^{-i\t_\E(\xi)x}\overline{g(x,\xi)}\right)+e^{i\t_\E(\xi)x}\partial_\xi g(x,\xi)-e^{-i\t_\E(\xi)x}\partial_\xi\overline{g(x,\xi)}\\
	&=i\t'_\E(\xi)x\left(e^{i\t_\E(\xi)x}g(x,\xi)-e^{-i\t_\E(\xi)x}\overline{g(x,\xi)}\right)+O(1).
	\end{align*}
	and
	\begin{align*}
	\partial_xv(x,\xi)=i\t_\E(\xi)e^{i\t_\E(\xi)x}g(x,\xi)-i\t_\E(\xi)e^{-i\t_\E(\xi)x}\overline{g(x,\xi)}+e^{i\t_\E(\xi)x}\partial_xg(x,\xi)+e^{-i\t_\E(\xi)x}\partial_x\overline{g(x,\xi)}
	\end{align*}
	and 
	\begin{align*}
	\partial_x\partial_\xi v(x,\xi)=i\t_\E'(\xi)x\left(i\t_\E(\xi)e^{i\t_\E(\xi)x}g(x,\xi)+i\t_\E(\xi)e^{-i\t_\E(\xi)x}\overline{g(x,\xi)}\right.\hspace{2cm}\\
	\left.+e^{i\t_\E(\xi)x}\partial_xg(x,\xi)-e^{-i\t_\E(\xi)x}\partial_x\overline{g(x,\xi)}\right)+O(1).
	\end{align*}
	Thus, having in mind \eqref{eq:Diagonal} we compute
	\begin{align*}
	\partial_\xi v(x,\xi)\partial_xv(x,\xi)-v(x,\xi)\partial_x\partial_\xi v(x,\xi)=2i\t'_\E(\xi)x\left(-2i\t_\E(\xi)|g|^2+W(g(\cdot,\xi),\overline{g(\cdot,\xi)})(x)\right)+O(1)
	\end{align*}
	Using Lemma \ref{lem:WronskianIdent} and \eqref{eq:rhoPrime} we see
	\[
	2i\t'_\E(\xi)x(-2i\t_\E(\xi)|g|^2+W(g(\cdot,\xi),\overline{g(\cdot,\xi)})(x))=\frac{2\pi if_\E(\xi)}{2\pi if_\mu(\xi)}x
	\]
	which finishes the proof by noting the the $O(1)$ is uniform for $\xi\in I$. 
\end{proof}
We will use later that this implies in particular  for any $\e>0$
\begin{align}\label{eq:1May27}
\lim_{L\to\infty}\frac{K_{(1+\e)L}(\xi,\xi)}{K_{L}(\xi,\xi)}=1+\e
\end{align}
uniformly for $\xi\in I$.

\section{Proofs of the Main theorems}\label{sec:Variational}
The goal of this section is to prove Theorem \ref{thm1},\ref{thm2} and \ref{thm3}.  We will need some preparatory work. 

Recall the set
\[
S_L=\left\{\int_0^L\cos(\sqrt{z}t)f(t)dt, \quad f\in L^2((0,L))\right\}
\]
and that 
\begin{align}\label{eq:minLambda}
\l_L(\xi)=\min\left\{\|F\|^2_{L^2_{d\mu}}: F\in S_L,\, F(\xi)=1\right\},
\end{align}
with minimizer
\begin{align}\label{eq:minimizer}
Q_L(z,z_0)=\frac{K_L(z,z_0)}{K_L(z_0,z_0)}.
\end{align}

\begin{lemma}\label{lem:funcionF}
	Let $d_0\in\bbR$ and $\xi_0>d_0$. Then there is $c>0$ such that 
	\begin{align*}
	F_c(z)=\frac{\sin(c(\sqrt{z-d_0}-\sqrt{\xi_0-d_0}))}{\sqrt{z-d_0}-\sqrt{\xi_0-d_0}}+\frac{\sin(c(\sqrt{z-d_0}-\sqrt{\xi_0-d_0}))}{\sqrt{z-d_0}+\sqrt{\xi_0-d_0}}
	\end{align*}
	satisfies 
	\begin{enumerate}[(i)]
		\item $|F_c(\xi)|\leq F_c(\xi_0)$ for all $\xi\geq d_0$;
		\item For any $\d>0$ there exists $\e>0$  such that for any $\xi\in[d_0,\infty)\setminus(\xi_0-\d,\xi_0+\d)$ it holds that $|F_c(\xi)|\leq F_c(\xi_0)-\e$;
		\item $F_c\in S_c$. 
	\end{enumerate}
\end{lemma}
\begin{proof}
	First we find $c$ so that $|F_c|$ has a global maximum at $\xi_0$. 
	Define
	\[
	G(u)=\frac{\sin(u-u_0)}{u-u_0}+\frac{\sin(u+u_0)}{u+u_0},
	\]
	where $u_0$ is the first positive root of $\tan(2u_0)-2u_0$. It can be directly verified $G$ is even and that $|G(u)|\leq G(u_0)$ for all $u\in \bbR$. From this it follows that for given $\z_0>0$, $c=u_0/\z_0$ and
	\[
	H(\z)=\frac{\sin( c(\z-\z_0))}{\z-\z_0}+\frac{\sin( c(\z+\z_0))}{\z+\z_0},
	\]
	that $|H|$ 	has a global maximum at $\z_0$. Since $\z(z)=\sqrt{z-d_0}$ maps $[d_0,\infty)$  onto $[0,\infty)$ the first claim follows by substitution and setting $\z_0=\z(\xi_0)=\sqrt{\xi_0-d_0}, c=u_0/\sqrt{\xi_0-d_0}$. It remains to show that $F\in S_c$. 
	
	We claim that $F_c$ is the reproducing kernel $K_c(z,\xi_0)$ for the constant potential $V=d_0$. The Christoffel Darboux formula \eqref{def:CDKernel} applied to $v(x,z)=\cos(x\sqrt{z-d_0})$ yields
	\[
	K_c(z,\xi_0)=\frac{-\sqrt{z-d_0}\cos(c\sqrt{\xi_0-d_0})\sin(c\sqrt{z-d_0})+\sqrt{\xi_0-d_0}\sin(c\sqrt{\xi_0-d_0})\cos(c\sqrt{z-d_0})}{\xi_0-z}.
	\]
	Using trigonometric identities we get that $K_c(z,z_0)$ is given by 
	\begin{align*}
	\frac{1}{2}\frac{\sin(c(\sqrt{z-d_0}-\sqrt{\xi_0-d_0}))(\sqrt{z-d_0}+\sqrt{\xi_0-d_0})+\sin(c(\sqrt{z-d_0}+\sqrt{\xi_0-d_0}))(\sqrt{z-d_0}-\sqrt{\xi_0-d_0})}{z-d_0-(\xi_0-d_0)}.
	\end{align*}
	Thus,
	\[
	F_c(z)=2K_c(z,\xi_0)
	\]
	and in particular, $F_c\in S_c$.
\end{proof}
\begin{remark}
	For later reference we mention that the $c$ is explicitly constructed in the proof. Let $u_0$ be the first positive solution of $2u=\tan(2u)$, i.e., $u_0$ is a constant not depending on $d_0,\xi_0$. Then $c$ is given by	
	\begin{align*}
	c=\frac{u_0}{\sqrt{\xi_0-d_0}}
	\end{align*}
	In particular, if $d_0(\e)=\min \E-\e$ and $\xi_0\in \Int(\E)$, then 
	\begin{align}\label{eq:cLimit}
	\lim\limits_{\e\to 0}\e c(\e)=0
	\end{align}
	and this limit is uniform for $\xi_0\in [a,b]\subset \Int(\E)$. 
\end{remark}

Recall the $\delta$-extension \eqref{eq:dextension}.
The following estimate is the crucial bound which allows to prove \eqref{eq:CFasymp} for regular potentials:
\begin{lemma}\label{lem:QEstimate}
	Let $V$ be a Stahl--Totik regular potential such that  $\E=\sigma_\ess(H_V)$ is Dirichlet regular, $\mu$ the associated spectral measure and $Q_L$ as in \eqref{eq:minimizer}. Then for any $\e>0$ there is $C>0$ and $\d>0$ such that for any $\xi\in \E_\d$ and $L>0$
	\[
	|Q_L(\xi,\xi_0)|\leq Ce^{\e L}\sqrt{\l_L(\xi_0)}.
	\] 
\end{lemma}
\begin{proof}

By Theorem \ref{thm:regBounds} we find $\d>0$ so that for $\xi\in\E_\d$ and $x>0$
	\[
	|v(x,\xi)|\leq \tilde Ce^{\e x}.
	\]
	Thus,
	\begin{align*}
	\int_0^L v(x,\xi)^2dx\leq \tilde C^2\int_0^L e^{2\e x}dx=\frac{\tilde C^2}{2\e}(e^{2\e L}-1)\leq \frac{\tilde C^2}{2\e}e^{2\e L}.
	\end{align*}
	Thus with $C^2=\tilde C^2/2\e$ we have 
	\[
	\left(\int_0^L v(x,\xi)^2dx\right)^{1/2}\leq Ce^{\e L}.
	\] 
	On the other hand, by the reproducing kernel property 
	\[
	|K_L(\xi,\xi_0)|=|\langle K_L(\cdot,\xi_0),K_L(\cdot,\xi)\rangle|\leq \| K_L(\cdot,\xi_0)\|\| K_L(\cdot,\xi)\|=\sqrt{ K_L(\xi_0,\xi_0)}\sqrt{ K_L(\xi,\xi)}
	\]
	Using
	\[
	K_L(\xi,\xi)=\int_0^L v(x,\xi)^2dx,\quad \l_L(\xi_0)=\frac{1}{K_L(\xi_0,\xi_0)}
	\]
	and \eqref{eq:minimizer} the claim follows. 
\end{proof}

We are now ready to prove the main comparising result that allows to lift the results from Section \ref{sec:FiniteGap} to arbitrary regular potentials.

\begin{theorem}\label{thm:VariationalPrinc}
	Let $V,\tilde V$ be potentials satisfying \eqref{L1locunif} and $\mu,\tilde\mu$ the associated spectral measures and $\E=\sigma_\ess(H_V),\tilde \E=\sigma_\ess(H_{\tilde V})$. Suppose that $V$ is a Stahl--Totik regular potential and $\E$ Dirichlet regular. Let $I$ be a closed interval such that $I\subset \Int(\E)\cap \Int(\tilde \E)$, $\mu$ and $\tilde\mu$ are absolutely continuous in a neighborhood of $I$ and its densities $f_\mu,f_{\tilde \mu}$ are positive and continuous at every point of $I$. For any $\e>0$, let $\d>0$ be as in Lemma   \ref{lem:QEstimate}. If there exists $\d_1>0$ such that
	\begin{align}\label{eq:EdeltaInclusion}
	\tilde\E_{\d_1}\subset \E_{\d}
	\end{align}
	then there exist $\d_2=\d_2(I)>0$, $D=D(\e)>0$ and $\g=\g(\e,\d_2)<1$ such that 
	\[
	\frac{\l_M(\xi_0,\tilde V)}{\l_L(\xi_0, V)}\leq \sup_{|\xi-\xi_0|<\d_2}\frac{f_{\tilde{\mu}}(\xi)}{f_\mu(\xi)}+D\g^{4N}e^{2\e L}+De^{2\e L}2^{-2N},
	\]
	where $M=L+2Nc$ and $N=N(\e)$ is sufficiently large. 
\end{theorem}
\begin{remark}
	Let us comment on the meaning of \eqref{eq:EdeltaInclusion}. In the proof we will need to estimate the extremizer for $\l_L(\xi_0, V)$ on $\supp\tilde\mu$. Close to the spectrum this can be done due to regularity by  Lemma \ref{lem:QEstimate}. However, there may be point masses of $\tilde\mu$ in the gaps. By extending $\tilde E$ we ensure that there are only finitely many point masses in $\bbR\setminus \tilde \E_{\d_1}$, since there can only be finitely many eigenvalues in each gap of $\tilde \E_{\d_1}$ and $\tilde \E_{\d_1}$ is a finite gap set. 
\end{remark}
\begin{proof}
	Let $Q^V_L(\xi,\xi_0)$ be the minimizer for $V$ and the point $\xi_0$. Then by Lemma \ref{lem:QEstimate} for $\xi\in \E_\d$ we have
	\begin{align}\label{eq:1May5}
	|Q^V_L(\xi,\xi_0)|\leq Ce^{\e L}\sqrt{\l_L(\xi_0,V)}.
	\end{align}
	Thus by assumption this also holds on $\tilde\E_{\d_1}$. Let $d_0=\min \tilde\E_{\d_1}$, $F_c$ as in Lemma \ref{lem:funcionF} and 
	\[
	G(z)=\frac{F_c(z)}{F_c(\xi_0)}.
	\] 
	Then 
	\begin{enumerate}[(i)]
		\item $G(\xi_0)=1$;
		\item for any $r>0$ there is $\g<1$ such that for every $\xi>d_0$ with $|\xi-\xi_0|>r$, $|G(\xi)|<\g$;
		\item\label{it:decayG} there is $C_1>0$ such that for $\xi>\xi_0+1$, $|G(\xi)|\leq \frac{C_1}{\sqrt{\xi-\xi_0}}.$
		\item $G\in S_c$.
	\end{enumerate}
	Since $\tilde\E_{\d_1}$ has only finitely many gaps, there are only finitely many point masses of $\tilde \mu$ in $\bbR\setminus\tilde\E_{\d_1}$. Let these  points be denoted by $\z_1,\dots,\z_n$. Let $P$ be a polynomial of degree $n+1$ that vanishes at these points and $P$ has a local maximum at $\xi_0$ such that $P(\xi_0)=1$. Let $N> n+1$ and define 
	\[
	Q=Q^V_LG^{2N}P. 
	\] 
	We claim that $Q\in S_{L+2Nc}$. Use $k^2=z$  and define
	\[
	H(k)=Q(k^2).
	\]
	Since $Q^V_L\in S_L$ and $F\in S_c$ it follows that $H$ is an even entire function of exponential type at most $L+2Nc$. Thus, by \eqref{eq:cSL} it remains to show that $H\in L^2_{dk}$. By \eqref{it:decayG} $G^{2N}P$ are bounded on $\bbR_+$. Moreover, since $Q^V_L\in S_L$, $Q_L(k^2)\in L^2_{dk}$ and we conclude that $H\in L^2_{dk}$. Moreover, we have $Q(\xi_0)=1$.  
	
	Thus, by \eqref{eq:minLambda} we get
	\begin{align}\label{eq:July6}
	\l_{L+2Nc}(\xi_0,\tilde V)\leq \|Q\|^2_{L^2_{d\tilde \mu}}.
	\end{align} 
	We will split the integral into several parts. First let $\d_2>0$ such that on $I_0=(\xi_0-\d_2,\xi_0+\d_2)$, $\mu,\tilde \mu$ are purely absolutely continuous and both are positive there. This can be achieved since they are continuous at every point of $I$. Moreover, let $\d_2$ be sufficiently small so that $|G|,|P|\leq 1$ on $I_0$.  Then
	\begin{align*}
	\int_{I_0}|Q(\xi)|^2d\tilde \mu(\xi)&\leq \int_{I_0}|Q_L^{V}(\xi,\xi_0)|^2d\tilde \mu(\xi)\\
	&\leq \sup_{t\in I_0}\frac{ f_{\tilde\mu}(t)}{f_\mu(t)}\int_{I_0}|Q^V_L(\xi,\xi_0)|^2d \mu(\xi)\\
	&\leq\sup_{t\in I_0}\frac{ f_{\tilde\mu}(t)}{f_\mu(t)}\l_L(\xi_0,V).
	\end{align*}
	Let us note that on $\supp(\tilde \mu)$ we have
	\begin{align}\label{eq:2May5}
	|Q(\xi)|\leq Ce^{\e L}\sqrt{\l_L(\xi_0)}|G(\xi)|^{2N}|P(\xi)|.
	\end{align}
	For, we have already argued that \eqref{eq:1May5} holds on $\tilde\E_{\d_1}$. Thus, the only points where \eqref{eq:2May5} may fail are the finite point masses of $\tilde \mu$ in $\bbR\setminus\tilde\E_{\d_1}$. But this is where $P$ vanishes and thus we obtain \eqref{eq:2May5} on $\supp(\tilde \mu)$. 
	
	Let $I_1=(\supp(\tilde \mu)\setminus I_0)\cap(-\infty, \xi_0+1]$ and $I_2=(\supp(\tilde \mu))\cap(\xi_0+1, \infty)$.
	Then
	\begin{align*}
	\int_{I_1}|Q(\xi)|^2d\tilde \mu(\xi)&\leq C^2e^{2\e L}\l_L(\xi_0,V)\int_{I_1} |G(\xi)|^{4N}|P(\xi)|^{2}d\tilde \mu(\xi)\\
	&\leq \g^{4N}C^2e^{2\e L}\l_L(\xi_0,V)\int_{I_1} |P(\xi)|^{2}d\tilde \mu(\xi)\\
	&=\g^{4N}C^2e^{2\e L}C_1\l_L(\xi_0,V).
	\end{align*}
	
	Since 
	\[
	\int_{\bbR}\frac{d\tilde\mu(\xi)}{1+\xi^2}<\infty,
	\]
	we get for $n\geq 2$
	\[
	\int_2^\infty\frac{d\mu(\xi)}{\xi^n}\leq \frac{K}{2^n},\quad K=4\int_2^\infty\frac{d\mu(\xi)}{\xi^2}.
	\]
	We conclude
	\begin{align*}
	\int_{I_2}|Q(\xi)|^2d\tilde \mu(\xi)&\leq C^2e^{2\e L}\l_L(\xi_0,V)\int_{I_2} |G(\xi)|^{4N}|P(\xi)|^{2}d\tilde \mu(\xi)\\
	&\leq C^2C_2e^{2\e L}\l_L(\xi_0,V)\int_{I_2} \frac{\xi^{2(n+1)}}{(\xi-\xi_0)^{2N}}d\tilde \mu(\xi)\leq C_3e^{2\e L}\l_L(\xi_0,V)2^{-2N}.
	\end{align*}
	Combining the integrals over $I_1,I_2,I_3$ and using \eqref{eq:July6} yields the claim.
\end{proof}
\begin{remark}
	In the proof we have $d_0=\min \tilde\E_{\d_1}$ and thus $d_0$ depends on $\e$. By the definition of $G$ via $F$ in Lemma \ref{lem:funcionF} this shows that $\g<1$, which is the maximum of $G$ outside of $(\xi_0-\d_2,\xi_0-\d_2)$, also depends on $\e$. However, for fixed $\d_2$ we have 
	\begin{align}\label{eq:tildeG}
	\g_1=\sup_{\e\in(0,1)}\g(\e,\d_2)<1.
	\end{align}
	This remains true, if $\xi_0\in [a,b]\subset\Int(\E)$. This will be important in the following. 
\end{remark}

We are now ready to prove Theorem \ref{thm1}:

\begin{proof}[Proof of Theorem \ref{thm1}]
	Assume that $\E=\sigma_{\ess}(\mu)$ and $\mu$ is regular and let $\E_r$ be the extension as defined above. Clearly $\E_r$ is a finite gap set and $\E_r$ decreases monotonically to $\E$. We conclude from Lemma \ref{lem:AppendixMartinMeasure} that $d\rho_{\E_r}$ is absolutely continuous on $I$ and that the densities $f_{\E_r}(\xi)$ increase with $r$ and are bounded above by $f_\E(\xi)$. We claim that 
	\begin{align}\label{measureIncreasing}
	\lim\limits_{r\to\infty}f_{\E_r}(\xi)=f_\E(\xi)
	\end{align}
	uniformly on $I$. By monotonicity and boundedness for every $\xi$, $\lim_{r\to\infty}f_{\E_r}(\xi)=g(\xi)$ exists and since $f_{\E_r}$ are in particular continuous the convergence is uniform and $g$ is continuous. On the other hand, using the upper envelope theorem, we conclude as in \cite[Lemma 6.1]{EichLuk} that $M_{\E_r}\to M_\E$, which implies $\rho_{\E_{r}}\to\rho_\E$ in the weak-$*$ sense. We conclude that $g=f_\E$ on $I$. 
	
	For fixed $r$ let $V_r$ be a finite gap potential as discussed in Section \ref{sec:FiniteGap} and $\mu_r$ its spectral measure. Note that $V_r$ is Stahl-Totik regular. We will apply Theorem \ref{thm:VariationalPrinc} with $\mu=\mu_r$ and $\tilde\mu =\mu$. Since $\E\subset \E_r$, \eqref{eq:EdeltaInclusion} is satisfied for arbitrary $\e>0$. Let $\e>0$ be fixed, and $\d_2,D$ be as in Theorem \ref{thm:VariationalPrinc} and $\g_1$ be defined by \eqref{eq:tildeG}. Then
	\[
	\frac{	\l_M(\xi_0, V)}{\l_L(\xi_0, V_r)}\leq \sup_{|\xi-\xi_0|<\d_1}\frac{f_\mu(t)}{f_{\mu_r}(t)}+D\g_1^{4N}e^{2\e L}+De^{2\e L}2^{-2N},
	\]
	where $M=L+2Nc$. Choose $\eta$ so that $\max\{\g_1, 1/\sqrt{2}\}^\eta\leq e^{-1}$ and $N=\eta \e L$. Note that by the definition of $\g_1$, $\eta$ does not depend on $\e$. It follows that
	\[
	D\g_1^{4N}e^{2\e L}+De^{2\e L}2^{-2N}=O(e^{-2\e L}).
	\] 
	Thus, 
	\[
	\limsup_{M\to\infty}\frac{	\l_M(\xi_0, V)}{\l_L(\xi_0, V_r)}\leq \sup_{|\xi-\xi_0|<\d_1}\frac{f_\mu(\xi)}{f_{\mu_r}(\xi)}.
	\]
	Since $M=L(1+2\eta \e c)$ we obtain by \eqref{eq:1May27} 
	\[
	\lim_{L\to\infty}\frac{\l_{L}(\xi_0,V_r)}{\l_{(1+2\eta \e c)L}(\xi_0,V_r)}=1+2\eta \e c.
	\]
	Therefore, 
	\[
	\limsup_{M\to\infty}\frac{	\l_M(\xi_0, V)}{\l_M(\xi_0,V_r)}\leq \sup_{|\xi-\z_0|<\d_1}\frac{f_\mu(\xi)}{f_{\mu_r}(\xi)}(1+2\eta \e c).
	\]
	Taking first $\e\to 0$ and using that $\e c(\e)\to 0$ by \eqref{eq:cLimit} and then $\d_1\to 0$ we get 
	\begin{align}\label{eq:2May26}
	\limsup_{M\to\infty}\frac{	\l_M(\xi_0, V)}{\l_M(\xi_0, V_r)}\leq \frac{f_\mu(\xi_0)}{f_{\mu_r}(\xi_0)}.
	\end{align}
	Since on the other hand $\E_r$ is a finite gap set we have 
	\[
	\lim\limits_{M\to\infty}M\l_M(\xi_0, V_r)=\frac{f_{\mu_r}(\xi_0)}{f_{\E_r}(\xi_0)}.
	\]
	Plugging this into \eqref{eq:2May26} yields 
	\begin{align*}
	\limsup\limits_{M\to\infty}M\l_M(\z_0, V)\leq \frac{f_\mu(\xi_0)}{f_{\E_r}(\xi_0)}. 
	\end{align*}
	By \eqref{measureIncreasing} sending $r\to 0$ we conclude
	\begin{align}\label{eq:4May26}
	\limsup\limits_{M\to\infty}M\l_M(\z_0, V)\leq \frac{f_\mu(\xi_0)}{f_{\E}(\xi_0)}. 
	\end{align}
	
	To get the opposite inequality we would like to switch the roles of $V$ and $V_r$. It can be seen from the proof of Theorem \ref{thm:VariationalPrinc} that in this case we can even take $\d_1=0$, since $\mu_r$ only has finitely many point masses outside $\E_r$. In Theorem \ref{thm:VariationalPrinc} we have to estimate the eigensolutions of $H_V$ on $\E_r$. This shows that for fixed $r>0$, we cannot take $\e\to 0$. However, for fixed $r>0$ we find $\e(r)$ and note that $\e(r)\to 0$ as $r\to 0$.
	
	Now switching the roles of $V$ and $V_r$ we get 
	\[
	\frac{	\l_M(\xi_0,V_r)}{\l_L(\xi_0, V)}\leq \sup_{|\xi-\xi_0|<\d_1}\frac{f_{\mu_r}(\xi)}{f_\mu(\xi)}+D\g_1^{4N}e^{2\e L}+De^{2\e L}2^{-2N},
	\]
	where $N=\eta \e L$ and as before 
	\[
	\limsup_{L\to\infty}\frac{\l_L(\xi_0, V_r)}{\l_L(\xi_0, V)}\leq\frac{f_{\mu_r}(\xi_0)}{f_\mu(\xi_0)}\frac{1}{1+2\eta \e c}
	\]
	and
	\[
	\limsup_{L\to\infty}\frac{1}{L\l_L(\z_0, V)}\leq\frac{f_{\E_r}(\xi_0)}{f_\mu(\xi_0)}\frac{1}{1+2\eta \e c}.
	\]
	But now for fixed $r$ we cannot take $\e\to 0$. However since $\e(r)\to 0$ as $r\to 0$ we get by \eqref{measureIncreasing}
	\[
	\limsup_{L\to\infty}\frac{	1}{L\l_L(\xi_0,V)}\leq\frac{f_{\E}(\xi_0)}{f_\mu(\xi_0)}
	\]
	and thus 
	\begin{align}\label{eq:3May26}
	\frac{f_\mu(\xi_0)}{f_{\E}(\xi_0)}\leq \liminf_{L\to\infty}L\l_L(\xi_0, V).
	\end{align}
	Combing \eqref{eq:4May26} and \eqref{eq:3May26} and noting that all the arguments are uniform in $\xi_0\in I$ yields the claim.
\end{proof}

\begin{proof}[Proof of Theorem \ref{thm2}]
Let $m$ be the Weyl $m$-function associated to $H_V$. We showed in Section \ref{sec:deBrangesSchroedinger} that $H_V$ can be written as a canonical system. Thus, by \cite[Theorem 9]{ELS} it follows that if for $\xi\in I$ the limit 
\begin{align}\label{eq:1Jan13}
\frac{1}{\pi}\lim_{z\to \xi}\Im m(z)=f_\mu(\xi)\in (0,\infty),
\end{align}
exists non-tangentially, then 
\begin{align}\label{eq:2Jan13}
\lim_{L\to\infty}\l_L(\xi)K_L(\xi+\l_L(\xi)z,\xi+\l_L(\xi)w)=\frac{\sin(\pi f_\mu(\xi)(\overline w-z))}{\pi f_\mu(\xi)(\overline w-z)},
\end{align}
uniformly for $\xi\in I$ and $z,w$ in compact subsets of $\bbC$. It follows from properties of Poisson integrals \cite[Theorem 11.22, Theorem 11.23]{RudinRealandComplex} that \eqref{eq:1Jan13} holds on $I$ under the assumptions of Theorem \ref{thm2}. From Theorem \ref{thm1} we conclude that 
\[
\lim_{L\to\infty}L\l_L(\xi)=\frac{f_\mu(\xi)}{f_\E(\xi)}.
\]
The claim follows from continuity of the sinc kernel and the fact that \eqref{eq:2Jan13} holds uniformly for $z,w$ in compact subsets of $\bbC$. 
\end{proof}

We finish this section with the proof of Theorem \ref{thm3}. There are several proofs in the orthogonal polynomials case that show how to conclude from universality clock spacing for the zeros of orthogonal polynomials, which only use interlacing properties of the zeros of orthogonal polynomials; cf. \cite{Freud,LevinLubinsky08,SimonCDKernel} The same proof carries over to the setting of continuum Schr\"odinger operators (or even canonical systems) without any change. We supply the proof for the reader's convenience.
\begin{proof}[Proof of Theorem \ref{thm3}]
We start by a well known fact: The function 
$$
m_L(z)=\frac{v(L,z)}{v'(L,z)}
$$
is a Herglotz function and since $v(L,\cdot), v'(L,\cdot)$ are entire, the measure in its integral representation is purely discrete and supported at the  zeros of $v'(L,z)$. Since $m_L$ is increasing between poles, the zeros of $v(L,z)$ and $v'(L,z)$ interlace.

By \eqref{def:CDKernel} and the fact that $v(L,\cdot)$ and $v'(L,\cdot)$ cannot vanish simultaneously, we see that for $z\neq \overline w$, $K_L(z,w)=0$ if and only if $m_L(z)=m_L(\overline w)$ (the value of $m_L(z)$ can also be $\infty$ corresponding to a zero of $v'(L,z)$). Fix $\xi\in I$ and define
\[
\tilde f_L(z)=\frac{K_L\big(\xi,\xi+\frac{z}{Lf_\E(\xi)}\big)}{K_L(\xi,\xi)}.
\]
Let $\dots<\tilde z^{L}_{-1}<\tilde z^{L}_0=0<\tilde z^{L}_1<\dots$ denote the zeros of $\tilde f_L(z)$. By \eqref{eq:universality} we see that 
$
\tilde z^{L}_{\pm 1}\to \pm 1
$ 
and inductively we get 
\[
\tilde z^{L}_{\pm j}\to \pm j.
\]
Let $0\leq z_0^{L}<\tilde z_1^{L}$ be so that $\xi+\frac{ z_0^{L}}{f_\E(x_0)L}$ is the first pole of $m_L$ to the right of $\xi$. Set
\[
f_L(z)=\frac{K_L\big(\xi+\frac{z_0^{L}}{Lf_\E(\xi)},\xi+\frac{z_0^{L}+z}{Lf_\E(\xi)}\big)}{K_L(\xi,\xi)}.
\]
If we denote the zeros of $f_L$ by $\dots<z^{L}_{-1}<z^{L}_0=0< z^{L}_1<\dots$ we see as before that 
\[
z^{L}_{\pm j}\to \pm j
\]
By our definition of $z_0^{L}$, $\xi_j^{L}=\xi+\frac{z_0^{L}+z^L_j}{f_\E(x_0)L}$ are the zeros of $v'(L,z)$, which finishes the proof. Since the convergence in \eqref{eq:universality} is uniform this shows that \eqref{spacing} holds uniformly on I.
\end{proof}
\appendix
\section{de Branges spaces and canonical systems}\label{app1}
It has already been realized in \cite{LubJFA} and also in \cite{ELS,BreuerConstr} that the theory of canonical systems is useful for the understanding of universality limits for Christoffel--Darboux kernels. The inverse theory developed by de Branges is based on a theory of Hilbert spaces of entire functions \cite{deBrangesHilbertSpace}. These spaces are called de Branges spaces. We recall some part of the general theory, to highlight that the objects discussed in Section \ref{sec:deBrangesSchroedinger} are only a special case of this rich theory. We follow the presentation in \cite{RemlingBookCanSys}.

For an entire function $F$ we denote $F^\#(z)=\overline{F(\overline{z})}$. Moreover, let $H^2=H^2(\bbC_+)$ denote the standard Hardy space of the upper half-plane. 
\begin{definition}
	A Hermite--Biehler function is an entire function $E$ with no zeros in $\bbC_+$ satisfying $|E(z)|\geq |E^\#(z)|$ for $z\in\bbC_+$. Given a Hermite--Biehler function $E$ we define the de Branges space $B(E)$ as
	\begin{align*}
	B(E)=\left\{F\mid F\text{ entire}, F/E, F^\#/E\in H^2\right\}.
	\end{align*}
\end{definition}
In the following it will be useful to decompose $E$ into its real and imaginary part. We define $a(z)=\frac{1}{2}(E(z)+E^\#(z))$ and $c(z)=\frac{1}{2i}(E^\#(z)-E(z))$. The notation $a$ and $c$ and the unexpected minus sign in the definition of $c$ will become clear below.

\begin{example}
	The motivating example for de Branges was the Hermite-Biehler function $E_a(z)=e^{-iaz}$ for $a>0$, in which case $B(E_a)$ denotes the standard Paley--Wiener space of entire square-integrable functions of exponential type at most $a$.
\end{example}
\begin{theorem}{\cite[Theorem 4.4]{RemlingBookCanSys}}
	Let $E$ be a Hermite--Biehler function. Then $B(E)$ becomes a Hilbert space when endowed with the scalar product
	\[
	\langle F,G \rangle_{B(E)}=\frac{1}{\pi}\int_{-\infty}^{\infty}F(t)\overline{G(t)}\frac{dt}{|E(t)|^2}.
	\]
	$B(E)$ is a reproducing kernel Hilbert space with reproducing kernel:
	\[
	K_E(z,w)=\frac{\overline{E(w)}E(z)-\overline{E^\#(w)}E^\#(z)}{2i(\overline{w}-z)}=\frac{a(z,x)\overline{c(w,x)}-c(z,x)\overline{a(w,x)}}{\overline{w}-z}.
	\]
\end{theorem}

De Branges spaces arise naturally when discussing canonical systems: 
\begin{definition}
	Consider  matrix-valued functions $A,B: [0,N) \to \Mat(2,\bbC)$ for some $N>0$ or $N=\infty$ which are locally integrable in the sense that their entries are in $L^1([0,x])$ for all $x  < N$, and have the property that $A(x)\geq 0,B(x)^*=B(x)$ and $\tr A(x)j= \tr B(x)j= 0$ for Lebesgue-a.e.\ $x\in [0,N)$, with $j$ defined below. A canonical system is a differential equation of the form 
	\begin{equation*} 
	j\partial_xy(x,z)=(-z A(x)+B(x))y(x,z),\quad j=\begin{pmatrix}
	0&-1\\1&0
	\end{pmatrix}.
	\end{equation*}
\end{definition}

A solution $T:[0,N)\times \bbC\to  \Mat(2,\bbC)$ which satisfies the initial value problem 
\begin{equation*}
j\partial_xT(x,z)= (-z A(x)+B(x))T(x,z),\quad T(0,z)=I_2
\end{equation*}
is called the transfer matrix of the canonical system. It is a fundamental object in the theory of canonical systems. 
Differentiating $T(s,w)^*jT(s,z)$ we see that 
\begin{align}\label{jform}
T(x,w)^*jT(x,z)-j=(\overline{w}-z)\int_0^xT(s,w)^*A(s)T(s,z)ds.
\end{align}
The expression on the left-hand side above is called the $j$ form of $T$ \cite{ArovDym,GolMik97}. For fixed $x$ as a function of $z$ it is entire and satisfies
\begin{align*}
i\left(T(x,z)^*jT(x,z)-j\right)=\begin{cases}
\geq 0,\quad &\text{if } z\in\bbC_+,\\
=0,\quad &\text{if } z\in\bbR.\\
\end{cases}
\end{align*} 
We say that $T$ is $j$-expanding in $\bbC_+$ and $j$-unitary on $\bbR$. Moreover, from \eqref{jform} it also follows that for $z\in \bbC_+$, $T$ satisfies the $j$-monotonicity property
\begin{align}\label{eq:jmonotonicity}
i(T^*(x_2,z)jT(x_2,z)-T^*(x_1,z)jT(x_1,z))\geq 0
\end{align}
for $x_2\geq x_1$. Thus, $\{T(x,z)\}_{x\in[0,N)}$ forms a $j$-monotonic family of entire matrix functions. It can be shown that vice-versa to every such family one can associate a canonical system \cite[Remark 2.3]{DEY21}. The $j$-form is invariant under multiplying $T$ from the left by some $U\in \SL(2,\bbR)$, since such $U$ satisfies $U^*jU=j$. This gives a certain freedom, which is called gauge freedom. A common gauge normalization, which was used by Potapov and de Branges is to assume that for all $x\in[0,N)$, $T(x,0)=I$, which leads on the level of canonical systems to $B=0$. In this case $A$ is usually denoted by $H$ and called the Hamiltonian of the system. We call this the Potapov--de Branges gauge. If $T$ is normalized arbitrarily, then passing to $\tilde T(x,z)=T(x,0)^{-1}T(x,z)$ we obtain a transfer matrix in the Potapov--de Branges gauge. The corresponding Hamiltonian is given by 
\[
H(x)=T(x,0)^*A(x)T(x,0).
\]
In particular, any canonical system obtained from a Schr\"odinger equation as in \eqref{eq:SchroedingerCanonical} can be rewritten into the Potapov--de Branges gauge. We found it more convenient to provide this gauge independent presentation and work directly with \eqref{eq:SchroedingerCanonical}.

Let us write
\[
T(x,z)=\begin{pmatrix}
a(z,x)& b(z,x)\\c(z,x)&d(z,x).
\end{pmatrix}
\]
It follows from \eqref{eq:jmonotonicity} that for fixed $x$
\[
m(z,x)=-\frac{a(z,x)}{c(z,x)}
\]
is a generalized Herglotz function, i.e., the map $z\mapsto m(z,x)$ maps $\bbC_+$ analytically into $\overline{\bbC_+}$, where $\overline{\bbC_+} = \bbC_+ \cup \bbR \cup \{\infty\}$; see \cite[Lemma 4.15]{RemlingBookCanSys}.\footnote{This definition differs slightly from the definition of Herglotz functions given in the introduction, since we allow values in $\overline{\bbC_+}$ rather than in $\bbC_+$. By the maximum principle, an analytic function that attains a value in  $\bbR \cup \{\infty\}$ must already be constant and therefore the set of generalized Herglotz function consists of Herglotz functions and constant functions with values in  $\bbR \cup \{\infty\}$.}

  Using this we see that 
 \begin{align}\label{eq:HB}
 	E(z,x):=a(z,x)-ic(z,x)
 \end{align}
 satisfies 
\[
\left|\frac{E^\#(z,x)}{E(z,x)}\right|=\left|\frac{a(z,x)+ic(z,x)}{a(z,x)-ic(z,x)}\right|=\left|\frac{-\frac{a(z,x)}{c(z,x)}-i}{-\frac{a(z,x)}{c(z,x)}+i}\right|\leq 1
\]
and thus $E(z,x)$ is a Hermite--Biehler function. This also explains in hindsight the minus sign in the definition of $c$. Dividing by $(\overline{w}-z)$ in \eqref{jform} the (1,1)-entry gives exactly the reproducing kernel 
\begin{align}\label{eq:CDFormula}
K_{E_x}(z,w)=\frac{a(z,x)\overline{c(w,x)}-c(z,x)\overline{a(w,x)}}{\overline{w}-z}=\begin{pmatrix}
	1\\0
\end{pmatrix}^*\int_0^xT(s,w)^*A(s)T(s,z)ds\begin{pmatrix}
1\\0
\end{pmatrix},
\end{align}
where $E_x(z)=E(z,x)$.

The canonical system is said to be limit point at $N$ if $\tr\int_0^N T(s,0)^*A(s)T(s,0)ds=\infty$. Due to \eqref{eq:jmonotonicity}, Weyl disks can be introduced in this setting. By a standard abuse of notation, we will use the same notation for an $\SL(2,\bbC)$ matrix and for the M\"obius transformation it generates on the Riemann sphere $\hat\bbC$, with the standard projective identification of $w\in\bbC$ with the coset of $\binom w1$ and $\infty$ with the coset of $\binom 10$. For any $z\in \bbC_+$, the Weyl disks are defined by
\[
D(x,z) = \{ w \in \hat\bbC \mid T(x,z) w \in \overline{\bbC_+} \}.
\] 
Due to \eqref{eq:jmonotonicity}, the Weyl disks are nested, $D(x_2, z) \subset D(x_1, z)$ for $x_1 \le x_2$. Thus, for each $z\in \bbC_+$, the intersection $\cap_{0\leq x<N} D(x,z)$ is a disk or a point. The assumption of being limit point at $N$ exactly means that this intersection is a point.
In this case, the Weyl disks define an analytic map $m: \bbC_+ \to \overline{\bbC_+}$ by
\begin{equation*}
\{ m(z) \} = \bigcap_{0\leq x<N}D(x,z).
\end{equation*}
Note that since $\SL(2,\bbR)$ matrices leave $\overline{\bbC_+}$ invariant, also $D(x,z)$ and thus $m$ do not depend on the gauge normalization.



\section{Martin measure}
Let $\E$ be a semibounded set so that for any Martin function  for the domain $\bbC\setminus\E$ and the point $\infty$ 
\[
\lim\limits_{z\to-\infty}\frac{M(z)}{\sqrt{-z}}>0.
\]
Sets with this property are called Akhiezer-Levin sets. Let again $M_\E$ be normalized so that the limit is equal to 1. Since $M_\E$ vanishes q.e. on $\E$ it can be extended to a subharmonic function on $\bbC$. Let $\rho_\E$ be the associated Riesz measure, defined by
\[
\rho_\E:=\frac{1}{2\pi}\Delta M_\E.
\]
Since $M_\E$ is a harmonic function in $\bbC_+$, we find an analytic function, $\Theta_\E$,  with $\Im \Theta_\E=M_\E$ and since $M_\E$ is positive, $\Theta_\E$ is a Herglotz function. Moreover, it can be shown that also $i\Theta'_\E$ is a Herglotz function and the measure in its integral representation is exactly $\rho_\E$:

\begin{lemma}{\cite[Lemma 2.3]{EichLuk}}
	The measure $\rho_\E$ is such that 
	\[
	\int_{\bbR}\frac{d\rho_\E(t)}{1+|t|}<\infty.
	\]
	Moreover, $i\Theta_\E'$ is a Herglotz function and we have 
	\begin{align*}
	i\Theta_\E'(z)=\int_{\bbR}\frac{d\rho_\E(t)}{t-z}.
	\end{align*}
\end{lemma}
Recall that $d\rho_\E(t)=f_\E dt+d\rho_{E,s}(t)$ denotes the Lebesgue decomposition of $\rho_\E$.
\begin{lemma}\label{lem:AppendixMartinMeasure}
	Let $(a,b)=I\subset \E^\circ$. Then  $\rho_\E(\xi)|_I$ is absolutely continuous and $f_\E(\xi)$ is real analytic. Moreover,  if $I\subset \E_1\subset \E_2$ we have
	\begin{align}\label{eq:monotonicityMartinMeasure}
	f_{\E_2}(\xi)|_I\leq f_{\E_1}(\xi)|_I.
	\end{align}
\end{lemma}
\begin{proof}
	Since $I$ only contains Dirichlet regular points by \cite[Theorem 4.2.2]{RansfordPotTheorie}, for every $\xi\in I$, $\lim\limits_{z \to \xi}M_\E(z)=0$, by \cite[Theorem 2.1]{EichLuk}. Since $\Theta_\E$ is a Herglotz function, this implies that it can be analytically extended through $I$. Hence, also $\Theta_\E'(z)$ has an analytic extension through $I$. Since $i\Theta_\E'$ is  a Herglotz function, this shows that the measure in its integral representation is purely absolutely continuous and moreover, 
	\[
	\frac{1}{\pi}\Re\Theta_\E'(\xi)=f_\E(\xi).
	\]
	Since $M_\E|_I=0$, $\Im\Theta_\E'(\xi)=0$ and we conclude that $f_\E$ is real analytic on $I$. 
	
	If $\E_1\leq \E_2$ we have that $M_{\E_2}\leq M_{\E_1}$ as follows e.g. from \cite[Lemma 2.7]{EichLuk}. Since by the Cauchy Riemann equations we get
	\[
	\pi f_\E(\xi)=\lim\limits_{y\to 0}\frac{M_\E(\xi+iy)}{y}
	\]
	and we conclude \eqref{eq:monotonicityMartinMeasure}.
\end{proof}

\providecommand{\MR}[1]{}
\providecommand{\bysame}{\leavevmode\hbox to3em{\hrulefill}\thinspace}
\providecommand{\MR}{\relax\ifhmode\unskip\space\fi MR }
\providecommand{\MRhref}[2]{%
	\href{http://www.ams.org/mathscinet-getitem?mr=#1}{#2}
}
\providecommand{\href}[2]{#2}

\end{document}